\documentclass[11pt, a4paper]{article}
\usepackage{amsmath,amssymb}
\usepackage[utf8]{inputenc}
\usepackage[colorlinks=true, citecolor=blue]{hyperref}
\usepackage{amsthm}
\usepackage{enumerate}
\usepackage{fullpage}
\usepackage{color}
\usepackage{graphicx}
\usepackage{float} 
\usepackage{subfigure}
\usepackage{mathtools}
\usepackage{enumitem}
\usepackage{tikz}
\usetikzlibrary{decorations.markings}
\numberwithin{equation}{section}
\usepackage{float}
\usepackage[capitalize,nameinlink,noabbrev]{cleveref}
\usepackage{array}
\usepackage{calc}
\newtheorem{theorem}{Theorem}
\newtheorem{lemma}[theorem]{Lemma}
\newtheorem{corollary}[theorem]{Corollary}

\newtheorem{remark}{Remark}

\usepackage{pifont}
\usepackage{mathtools}
\usepackage{mathrsfs}
\usepackage{array}
\usepackage{verbatim}
\usepackage{esvect}
\usepackage[affil-it]{authblk}

\providecommand{\keywords}[1]
{
  \small	
  \textbf{\textit{Keywords---}} #1
}

\title{Integrability of Combinatorial Riemann Boundary Value Problem and Lattice Walks in Three Quadrants}
\author{Ruijie Xu\thanks{\href{mailto:xuruijie@bimsa.cn}{xuruijie@bimsa.cn}}}

\affil{\small Beijing Institute of Mathematical Science and Applications (BIMSA),\\No. 544, Hefangkou Village Huaibei Town, Huairou District Beijing 101408}
\date{ } 
         
\begin{document}
\maketitle
\begin{abstract}
We introduce a general framework of matrix-form combinatorial Riemann boundary value problem (cRBVP) to characterize the integrability of functional equations arising in lattice walk enumerations. A matrix cRBVP is defined as integrable if it can be reduced to enough polynomial equations with one catalytic variable. Our central results establish that the integrability depends on the eigenspace of some matrix associated to the problem. For lattice walks in three quadrants, we demonstrate how the obstinate kernel method transforms a discrete difference equation into a $3\times 3$ matrix cRBVP. The special double-roots eigenvalue $1/4$ yields two independent polynomial equations in the problem. The other single-root eigenvalue yields a linear equation. We obtain three independent equations from a $3\times 3$ system. Crucially, our framework generalizes three-quadrant walks with Weyl symmetry to models satisfying only orbit-sum conditions. It explains many criteria about the orbit-sum proposed by various researchers and it also explains the counter-example of lattice walks starting outside the quadrant.
\end{abstract} \hspace{10pt}


\tableofcontents
\section{Introduction}
In algebraic or analytic combinatorics, many problems can be reduced to solving some functional equations. We discuss a special type of functional equation which appears widely in combinatorics. We call it the combinatorial version of the Riemann boundary value problem (cRBVP) due to its strong connection with the Riemann boundary value problem with Carleman shift \cite{xu2023combinatorial,fayolle2012some}. In general, the functional equation takes the following form,
\begin{align}
    H(1/x,t)=G(x,t)H(x,t)+C(x,t),\label{generic cRBVP}
\end{align}
where $G(x,t)$ and $C(x,t)$ are two known functions of $x$ and $t$. $H(x,t)$ and $H(1/x,t)$ are unknown functions. As an equation appearing in combinatorics, $H(x,t)$ is defined as the generating function of some combinatorial objects and we are looking for solutions as formal series of $t$ with polynomial coefficients in $x$.

The form \eqref{generic cRBVP} arises from the obstinate kernel method approach \cite{bousquet2005walks} for some 2-dimensional lattice walk problem. Mathematically speaking, a 2-D lattice walk model can be viewed as the simplest realization of the following linear discrete difference equation,
\begin{align}
    F(x,y,t)=P(x,y,t)+t\sum_{k,l}P_{kl}(x,y,t)\Delta_x^{(k)}\Delta_y^{(j)}F(x,y,t)\label{DDE},
\end{align}
where $P(x,y,t),P_{kl}(x,y,t)$ are known functions. $F(x,y,t)$ is the unknown and $\Delta_x$ ($\Delta_y$) is the discrete derivative with respect to $x$ ($y$)
\begin{align}
    \Delta_x: F(x,y,t)\to \frac{F(x,y,t)-F(0,y,t)}{x}.
\end{align}
The operator $\Delta_x^{(i)}$ is obtained by applying $\Delta_x$ $i$ times.

By the obstinate kernel method \cite{mishna2009classifying}, we can eliminate the tri-variate unknown function $F(x,y,t)$ and derive, for example, equations of $F(x,0,t)$ in the form of \eqref{generic cRBVP}.

\subsection{Historical Context of Lattice Walks}
Lattice walks serve as abstract models for many problems in different fields, including probability theory \cite{malyshev1972analytical}, condensed matter physics \cite{brak2005directed}, integrable systems \cite{Tong2021} and representation theory\cite{postnova2021counting}. The objective of this study is to find the explicit number of $n$-step path (the number of configurations). By some combinatorial construction, such problem becomes solving generating functions (of the number of configurations) satisfying some functional equations \eqref{DDE}.

The most widely studied models in lattice walks are the lattice walks in quarter-plane with small steps. Small steps means that each step is unit length in the following directions $\{\uparrow,\downarrow,\leftarrow,\rightarrow,\nearrow,\nwarrow,\searrow,\swarrow\}$. In \cite{bousquet2010walks,mishna2009classifying}, the authors clasified all $256$ possible small step walks into $79$ different non-trivial two-dimensional models. Among these $79$ models, $23$ are associated with finite symmetry groups and can be further classified: $16$ models correspond to $D_2$ groups, five to $D_3$ groups, and two to $D_4$ groups. In \cite{bousquet2010walks,bousquet2005walks,bousquet2016elementary,mishna2009classifying}, the authors solved all these models using the algebraic kernel method and the obstinate kernel method. Quarter-plane lattice walk models can also be reduced to a Riemann boundary value problem (RBVP) with Carleman shift\cite{Litvinchuk2000} and were solved using the conformal gluing function \cite{raschel2012counting}. Another approach involves Tutte's invariants \cite{raschel2020counting}. In \cite{xu2023combinatorial}, we established a combinatorial equivalence between these three approaches, demonstrating their fundamental consistency. 

The analysis becomes more complex for walks in three quadrants (also called lattice walks avoiding a quadrant). Initial studies date back to \cite{bousquet2016square}. By the kernel method, the authors successfully solved the Weyl models \cite{bousquet2023walks,bousquet2021more}. These are the models the quarter-plane walks of which can be solved by the reflection principle \cite{Gessel1992}. The RBVP approach \cite{raschel2018walks} and the Tutte's invariant approach \cite{bousquet2023enumeration} only work for models with diagonal reflection symmetry. For models lacking reflection symmetries, there is little progress from these three different approaches. In \cite{price2022enumeration}, the author provided a remarkable approach to three-quadrant lattice walks using elliptic functions and gives exact integral expressions for all cases. The results can also be extended to 2-D walks in $M$-quadrant cones for any positive integer $M$ (for $M=1$, it is the quarter-plane. $M=3$ is the three-quadrants. For $M>4$, we shall consider the cones on the Riemann surface).

\subsection{Objective of This Paper}
One of the main objectives of this paper is to extend the obstinate kernel method introduced in \cite{bousquet2016square} to more three-quadrant models. The advantage of this method lies in its ability to reveal transparent algebraic structures. For example, it was shown in \cite{bousquet2016square} that the generating functions of simple lattice walks (walks with allowed steps $\{\uparrow,\downarrow,\leftarrow,\rightarrow\}$) in three quadrants are expressed as the sum of the corresponding generating function of the same walk in the quarter-plane and an algebraic function. For example,
\begin{align}
    F(1/x,0,t)=-\frac{1}{3}x^2Q(1/x,0,t)+M(1/x,t).\label{catalytic pre}
\end{align}
Here $F(1/x,0,t)$ represents the generating function of walks ending on the negative horizontal axis in the three-quadrants model and $Q(1/x,0,t)$ corresponds to the generating function of walks ending on the positive horizontal axis in the quarter-plane model. $M(x,t)$ is an algebraic function satisfying a polynomial equation,
\begin{align}
    P(M(x,t),M_1,M_2\dots M_n,x,t)=0,\label{catalytic}
\end{align}
where $M_1,M_2,\dots,M_n$ are parameters independent of $x$. \eqref{catalytic} is a polynomial equation in one catalytic variable. It can be solved via the general strategy proposed in \cite{bousquet2006polynomial}.

The authors of \cite{bousquet2016square} derived this algebraic structure from the orbit-sum property in the algebraic kernel method. \eqref{catalytic pre} holds since the orbit-sum of the three-quadrant walk is section-free (i.e., orbit-sum does not contain unknown functions in $x$ or $y$) and equals the orbit-sum of the corresponding walk in the quarter plane. They further provide a combinatorial proof through the reflection principle \cite{guy1992lattice}. We will discuss these concepts later in the calculations. Here we emphasize that the orbit-sum property also arises in the non-Weyl models. In \cite{bousquet2023walks}, \textbf{Section 3.2}, the author noted that for non-Weyl models, constructing $M(x,t)$ satisfying \eqref{catalytic pre} is infeasible due to the lack of symmetries. However, the existence of \eqref{catalytic} remains open. In this work, we establish \eqref{catalytic} by constructing a matrix cRBVP from the functional equation \eqref{DDE}.

\subsection{Integrability for Walks in Three Quadrants and the cRBVP}
To avoid case-by-case proofs, we select a typical three-quadrant model exhibiting the orbit-sum property but lacking Weyl symmetry. The generating function of this model satisfies a linear discrete difference equation of the form \eqref{DDE}. By applying the obstinate kernel method, we show that \eqref{DDE} can be transformed into a matrix-type cRBVP \eqref{generic cRBVP} with specific algebraic properties. These properties imply that there exists a function $A(x,t)$ that satisfies a polynomial equation with one catalytic variable.

To generalize the result, we discuss the algebraic properties of matrix type cRBVP in a generic setting. The matrix cRBVP arising in a lattice walk problem assumes the following elliptic structure,
\begin{align}
H(1/x,t)=\left(P_0(x,t)+P_1(x,t)\sqrt{\Delta(x,t)}\right)H(x,t)+C(x,t),\label{independent 1}
\end{align}
where $P_0(x,t)$ and $P_1(x,t)$ are matrices with rational entries in $x,t$ and $\Delta(x,t)$ is the discriminant of some quadratic polynomial.

Matrix-type cRBVPs are not limited to three-quadrant lattice walks. For models with diverse boundary conditions including $M$-quadrant cones, the construction of a cRBVP from the discrete difference equation via the obstinate kernel method is a universal procedure.

The solvability condition is straightforward: if there exist $n$ functions $A_1(x,t),\dots,A_n(x,t)$ satisfying $n$ independent solvable equations, then the $n$-dimensional matrix cRBVP admits a solution. By reducing the system \eqref{DDE}, which involves multiple unknown functions such as $F(x,y,t)$ and $F(0,y,t)$, to a single equation with one unknown function $A(x,t)$, \eqref{catalytic} parallels the concept of first integrals in PDE theory or integrals of motion in classical mechanics. Polynomial equations with catalytic variables are also called discrete difference equations.

The solvability of a multivariate DDE (or equivalently, a cRBVP) separates into two phases: integrability and explicit solution. We call a lattice walk model or matrix cRBVP integrable if one can find enough polynomial equations with one catalytic variable (potentially including D-finite terms, as discussed later). However, as with the classical notions of integrability, solving these equations explicitly requires distinct techniques. This paper focuses on establishing integrability, leaving explicit solutions for future work.

\subsection{Structure and Results of This Paper}
This paper is organized as follows,
\begin{enumerate}
    \item In \cref{three quarter}, we analyze a three-quadrant lattice walk model lacking reflection symmetry and establish its exact solvability. The model has allowed steps $\{\nearrow,\nwarrow,\downarrow\}$. In \cref{statement}, we demonstrate that the solvability of this model is determined by a linear equation and two polynomial equations with one catalytic variable. The integrability follows from the independence of these three equations. We further prove the existence and uniqueness of solutions to these equations.  
    \item In \cref{RBVP}, we develop a general framework for $3 \times 3$ matrix cRBVPs, establishing their integrability and characterizing their algebraic structures. Integrability depends on the eigenvalues of an associated matrix, and the results are extended to higher-dimensional systems. A concise summary is provided in \cref{conclude}.  

    \item In \cref{orbit example}, we analyze a counterexample to the conjecture that zero orbit sums imply algebraic generating functions \cite{bousquet2016square,bousquet2016elementary,raschel2020counting}. The model is a whole-plane lattice walk with steps $\{\leftarrow,\rightarrow,\uparrow,\downarrow\}$, starting from $(-1,-1)$ and restricting transitions from the first quadrant to the other three quadrants. Introduced in \cite{Buchacher2022}, this model exhibits D-algebraic generating functions despite a zero orbit-sum. Using the matrix cRBVP framework, we reveal the structural mechanisms behind this discrepancy and prove the model's integrability. 
\end{enumerate}

\subsection{Notations for Formal Power Deries}
Before analyzing lattice walk models and cRBVPs (combinatorial Riemann boundary value problems), we introduce the following conventions for formal power series.

A fractional formal power series in \( x \) is defined as:
\begin{align}
    f(x) = \sum_{k \geq k_0} f_k x^{k/d},
\end{align}
where \( d \in \mathbb{Z} \setminus \{0\} \). We consider fractional power series since we may take algebraic roots of some series. We denote \( [x^i]f(x) \) as the coefficient of $i$th degree term of $x$ in \( f(x) \). We denote \( [x^>]f(x) \) as the positive-degree terms of \( x \) in $f(x)$, \( [x^<]f(x) \) as the negative-degree terms of \( x \) and \( [x^\geq]f(x) \) as the nonnegative-degree terms of \( x \).

Let \( \mathbb{K} \) be a commutative ring and \( \overline{\mathbb{K}} \) its algebraic closure. We define:
\begin{enumerate}
    \item \( \mathbb{K}[t] \): Polynomials in \( t \) over \( \mathbb{K} \).
    \item \( \mathbb{K}\left[t, \frac{1}{t}\right] \): Laurent polynomials in \( t \) over \( \mathbb{K} \).
    \item \( \mathbb{K}^{\text{fr}}[[t]] \): Fractional power series in \( t \) over \( \mathbb{K} \).
    \item \( \mathbb{K}^{\text{fr}}((t)) \): Fractional Laurent series in \( t \) over \( \mathbb{K} \).
    \item \( \mathbb{K}(t) \): Rational functions in \( t \) over \( \mathbb{K} \).
\end{enumerate}
These notations generalize iteratively to multivariate series. For instance: \( \mathbb{R}(x)[[t]] \): Laurent series in \( t \) with coefficients in \( \mathbb{R}(x) \).

We classify functions based on their algebraic and differential properties:
\begin{enumerate}
    \item \textbf{Algebraic}: \( f(x) \) is algebraic over \( \mathbb{C}(x) \) if \( \exists P(f, x) = 0 \), where \( P \) is a polynomial with coefficients in \( \mathbb{C}(x) \).
    \item \textbf{D-finite (Holonomic)}: \( f(x) \) is D-finite if \( \exists L(f, f', \dots, f^{(n)}) = 0 \), where \( L \) is a linear differential operator with coefficients in \( \mathbb{C}(x) \).
    \item \textbf{D-algebraic (Hyperalgebraic)}: \( f(x) \) is D-algebraic if \( \exists P(f, f', \dots, f^{(n)}) = 0 \), where \( P \) is a polynomial differential operator with coefficients in \( \mathbb{C}(x) \).
    \item \textbf{Hyper-transcendental}: \( f(x) \) satisfies none of the above.
\end{enumerate}

\begin{figure}[ht!]
\centering
\includegraphics[scale=0.25]{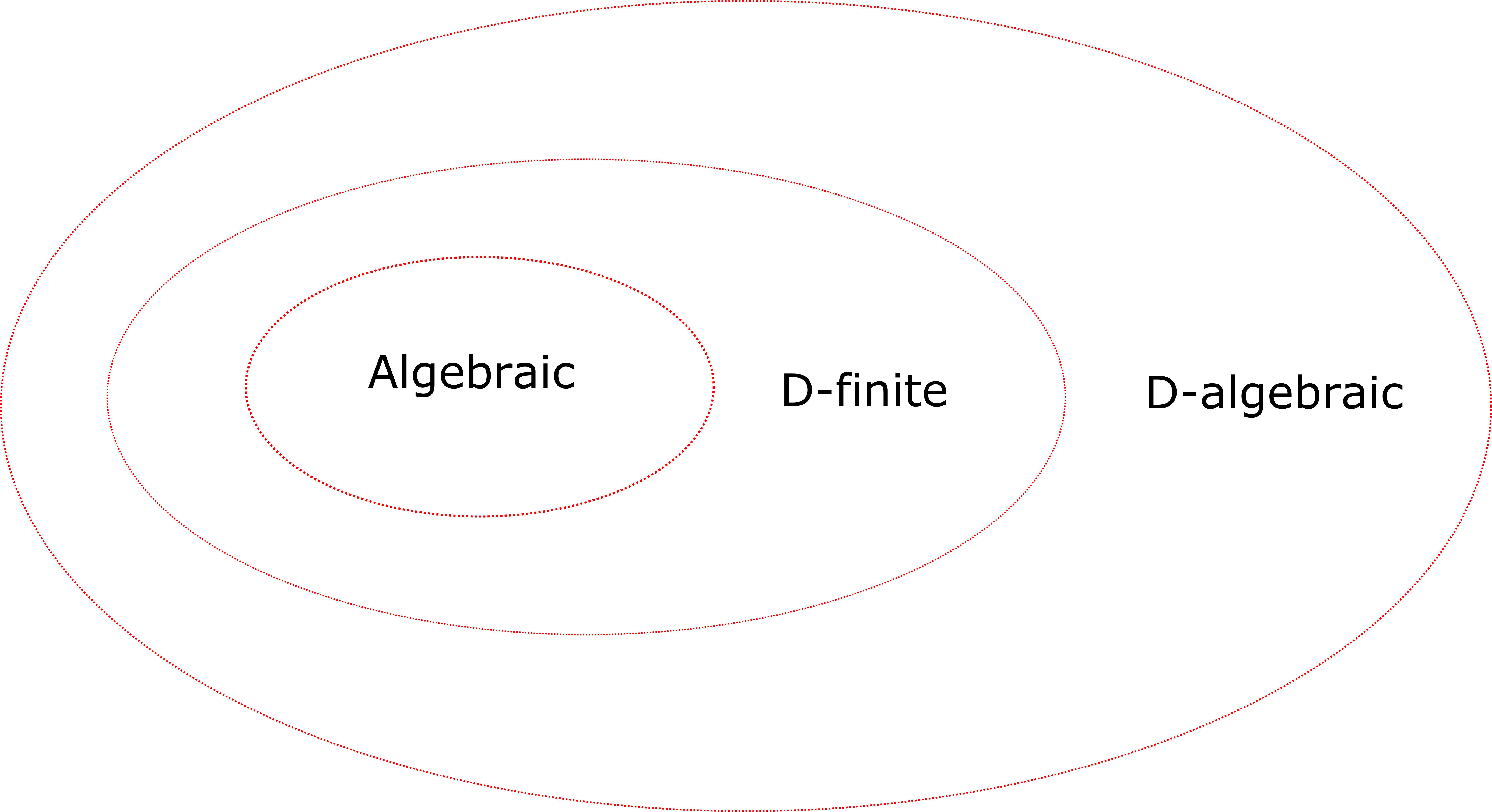}
\caption{The relation between algebraic, D-finite and D-algebraic}
\label{fig The relation between algebraic, D-finite and D-algebraic}
\end{figure}

For multivariate functions \( f(x_1, \dots, x_k) \),
\begin{itemize}
    \item \( f \) is D-finite if it satisfies linear differential equations in each variable \( x_i \) with coefficients in $\mathbb{C}(x_1, \dots, x_k)$.
    \item \( f \) is algebraic/D-algebraic if it satisfies polynomial (differential) equations in each \( x_i \) with coefficients in $\mathbb{C}(x_1, \dots, x_k)$.
\end{itemize}
For simplicity in the discussion, we may specify D-finite (algebraic, D-algebraic) over one variable $x_i$ if they satisfy the corresponding one-variable condition with coefficient in $\mathbb{C}(x_i)$. Notice that $f(x_1,x_2\dots x_k)$ is D-finite over each $x_i$ does not mean that it is a D-finite function.

\subsection{Notations for Lattice Walk Models}\label{def}
Lattice walks provide a foundational framework for modeling discrete difference equations. The recurrence relation of lattice walk is exactly a discrete Laplacian operator \cite{trotignon2022discrete,hoang2022discrete}. We construct the functional equation directly by the generating functions.
\begin{enumerate}
    \item The walk occurs on the $(i,j)$-plane. An arbitrary point is denoted as $(i,j)$. $x,y$ are the auxiliary variable for generating functions.
    \item The allowed step set $\mathcal{S}$ satisfies \[\mathcal{S} \subseteq \{-1,0,1\}\times \{-1,0,1\}\setminus \{(0,0)\}.\]
    The \emph{step generator} $S$ is defined as
\begin{align}
S(x,y) = \sum_{(k,l)\in\mathcal{S}}x^k y^l.
\end{align}
This can be regarded as the generating function of each step.
    \item The number of configurations, $f_{i,j,n}$ refers to the number of $n$-step paths from $(0,0)$ to $(i,j)$. The generating function is
\begin{align}
F(x,y,t)\equiv F(x,y)=\sum_{i,j,n}f_{i,j,n}x^i y^j t^n\equiv\sum_{i,j}F_{i,j}x^iy^j\label{function}.
\end{align}
We abbreviate $t$ in the notation. For instance, $F_{i,j}$ here is a formal series of $t$.
    \item The weight of a path is the product of the weight of each step. The default weight is $1$. We may add different weights to some special steps to change the symmetry of the model. For example, if we add weight $a$ to each step that visits $(0,0)$, \eqref{function} becomes,
\begin{align}
 F(x,y)=\sum_{i,j,n,k}f_{i,j,n,k}x^i y^j a^kt^n=\sum_{i,j}F_{i,j}(a)x^iy^j,
\end{align}
where $f_{i,j,n,k}$ is the number of paths of length $n$ that start at $(0,0)$, end at $(k,l)$, and visits $(0,0)$ $k$ times.
\item  We define the generating functions of walks ending on some lines. For example, the generating function of walks ending on the nonnegative $i$-axis,
\begin{equation}
[x^\geq y^0]F(x,y)\equiv F(x,0)=\sum_{i,n}f_{i\geq 0,n}x^it^n.
\end{equation}
\end{enumerate}
Other boundary terms will be defined during the calculations.

\subsection{Derivation of the Functional Equation}\label{Model}
To construct a functional equation, we obey the following logic: A lattice path is constructed by appending one step to a one-step shorter path. In the representation of generating functions, we have
\begin{align}
    F(x,y)= t S(x,y) F(x,y)+\text{boundary terms}.
\end{align}
Boundary terms are those steps which do not satisfy the definition of $F(x,y)$ or do not satisfy the general recursion relation.

For example. Suppose that we are considering a simple lattice walk ($\{\uparrow,\downarrow,\leftarrow,\rightarrow\}$) starting from $(0,0)$ and restricted in the right half-plane (including the $j$-axis). We defined $F(x,y)$ as the generating function of the paths ending in the right half-plane and $Q(x,y)$ as the generating function of the paths ending in the first quadrant (including the axis $j=0, i\geq 0$ and $i=0, j\geq 0$). The boundary terms for $F(x,y)$ are,
\begin{enumerate}
    \item $1$, refers to the first step. This is the initial step which does not satisfy the recursion relation.
    \item $-tF(0,y)/x$, which counts the illegal $\leftarrow$ on the $j$-axis.
\end{enumerate}
Then $F(x,y)$ satisfies the following functional equations,
\begin{align}
    F(x,y)=t\left(x+\frac{1}{x}+y+\frac{1}{y}\right)F(x,y)+1-tF(0,y)/x\label{func right half}.
\end{align}
The boundary terms for $Q(x,y)$ are
\begin{enumerate}
    \item $1$, refers to the first step.
    \item The illegal $\leftarrow$ on the positive $j$-axis. We denote them as $-tVp(y)/x=-\frac{t}{x}[y^\geq]F(0,y)$.
    \item The steps that quit the first quadrant. This is $-\frac{t}{y}[x^\geq y^0]F(x,y)$ ($\downarrow$ on line $j=0$). We denote them as $-\frac{t}{y}Hp(x)$.
    \item The steps that enter the first quadrant from the outside. This is $ty y^{-1}[x^\geq y^{-1}]F(x,y)$ ($\uparrow$ on line $j=-1$. Notice that the definition $[y^{-1}]f(y)$ only refers to the coefficients. The generating function of $y^{-1}$ terms is $y^{-1}[y^{-1}]f(y))$. We denote them as $tyHp_{-1}(x)/y$.
\end{enumerate}
(Naming convention: $Hp_{a}$ refers to the horizontal positive line $j=a, i\geq 0$. If $a=0$, we omit $a$ and write $Hp(x)$. $Vp$ refers to vertical positive). Similarly, $Hn$ represents `horizontal negative'. By convention, $0$ is included in the positive part. To simplify notation, we denote $1/x$ as $\bar{x}$ and $1/y$ as $\bar{y}$ throughout the analysis.

The functional equation for $Q(x,y)$ reads,
\begin{align}
    Q(x,y)=1+t(x+\bar{x}+y+\bar{y})Q(x,y)-t\bar{y}Hp(x)+tHp_{-1}(x)\label{upper half plane g}.
\end{align}
By the recursion relation, we can construct the functional equation of any generating function for lattice walk problems.

\eqref{upper half plane g} can also be obtained by taking $[y^{\geq}]$ terms of \eqref{func right half}. $Q(x,y)$ is a part of $F(x,y)$.

In the following sections, we present functional equations for diverse lattice walk models without explicit proofs, as they are systematically derived via the methodology outlined in this section.

\section{Walks Avoiding a Quadrant with Full Orbit-Sum  Properties}\label{three quarter}
We consider a model without Weyl symmetry, walk with allow steps $\{\nearrow,\nwarrow,\downarrow\}$ in three quadrants. This walk only has $x\to \bar{x}$ symmetry. Following the idea in \cite{bousquet2016square,bousquet2023walks}, we start this section by considering the orbit-sum and applying the algebraic kernel method \cite{bousquet2005walks}.
\subsection{Walks with Allowed Steps $\{\nearrow,\nwarrow,\downarrow\}$}
The generating function of walk in three quadrants is denoted as $F(x,y)$. The boundary terms $Hp(x)$, $Hn(\bar{x})$, $Vp(y)$, $Vn(\bar{y})$ are defined as per in \cref{Model}, which is horizontal positive, horizontal negative, vertical positive and vertical negative.

Moreover, we allow an extra $\{\nwarrow\}$ from $(0,-1)\to (-1,0)$ with weight $p$. If we choose $p=0$, this is the original three-quadrant model. In the calculation later, we show that this special step does not affect the integrability, but taking $p=1$ will simplify the calculation.

The functional equation of $F(x,y)$ reads,
\begin{align}
    (1-t(x y+y\bar{x}+\bar{y}))F(x,y)=1-t\bar{y}HN(\bar{x})-ty\bar{x}VN(\bar{y})+tp\bar{x}F_{0,-1}\label{model3 func}.
\end{align}
the kernel is defined as $K(x,y)=(1-t(x\bar{y}+y\bar{x}+\bar{y}))$ and it is invariant under the following two involutions
\begin{align}
    \phi :(x,y)\to(1/x,y) \qquad \psi:(x,y)\to \left(x,\frac{1}{y(x+\bar{x})}\right).\label{group for our model}
\end{align}
$\phi$ and $\psi$ generates a $D_2$ group $G$,
\begin{align}
(x,y)\to (1/x,y)\to \left(1/x,\frac{1}{y(x+\bar{x})}\right)\to \left(x,\frac{1}{y(x+\bar{x})}\right).
\end{align}
\begin{figure}[h]
\centering
\tikzstyle{element} = [rectangle, rounded corners, minimum width = 1cm, minimum height=0.8cm,text centered, draw = black]
\tikzstyle{arrow} = [->,>=stealth]
\begin{tikzpicture}[node distance=1cm]
\node[element](e00){$J_0:(x,y)$};
\node[element, below of =e00, yshift = -0.2cm,  xshift = -2cm](e10){$J_1:(\bar{x},y)$};
\node[element, below of =e00, yshift = -0.2cm,  xshift = 2cm](e01){$J_2:(\bar{x},\frac{1}{y(x+\bar{x})})$};
\node[element, below of =e10, yshift = -0.2cm, xshift=2cm](e44){$J_3:(x,\frac{1}{y(x+\bar{x})})$};
\draw [arrow] (e00) -- (e01);
\draw [arrow] (e00) -- (e10);
\draw [arrow] (e01) -- (e44);
\draw [arrow] (e10) -- (e44);
\end{tikzpicture}
    \label{fig:enter-label}
\end{figure}

\subsection{The Full Orbit-Sum}
If we use $J_0,J_1,J_2,J_3$ to denote the functional equation \eqref{model3 func} after applying the corresponding transformations in the group (Fig.\ref{fig:enter-label}), then the alternating sum $xy J_0-y\bar{x}J_1+\frac{\bar{x}}{(x+\bar{x})y}J_2-\frac{x}{(x+\bar{x})y}J_3$ is section free and reads,
\begin{align}
\begin{split}
    &\left(xy F(x,y)-y\bar{x}F(\bar{x},y)+\frac{\bar{x}}{(x+\bar{x})y}F\left(\bar{x},\frac{1}{y(x+\bar{x})}\right)-\frac{x}{(x+\bar{x})y}F\left(x,\frac{1}{y(x+\bar{x})}\right)\right)\\
    &=\frac{(x-1) (x+1) \left(x^2 y^2-x+y^2\right)}{x \left(x^2+1\right) yK(x,y)}\label{full orbit sum Model 3}.
\end{split}
\end{align}
The sum of all $g(xyJ_0),g\in G$ with suitable coefficients in $\mathbb{C}(t)(x,y)$\footnote{For walks without boundary weights, the coefficients are $\pm 1$ alternatively.} is called the orbit-sum ($OS$). We call the orbit-sum section free if the sum eliminates all boundary unknown functions on the right hand-side. If we consider the walk with same allowed steps but restricted in the quadrant and denote the generating function as $Q(x,y)$, the full orbit-sum of $Q(x,y)$ reads,
\begin{align}
\begin{split}
    &xy Q(x,y)-y\bar{x}Q(\bar{x},y)+\frac{\bar{x}}{(x+\bar{x})y}Q\left(\bar{x},\frac{1}{y(x+\bar{x})}\right)-\frac{x}{(x+\bar{x})y}Q\left(x,\frac{1}{y(x+\bar{x})}\right)\\
    &=\frac{(x-1) (x+1) \left(x^2 y^2-x+y^2\right)}{x \left(x^2+1\right) yK(x,y)}\label{full orbit sum Model 32}.
\end{split}
\end{align}
The right hand-sides of the orbit-sum of $F(x,y)$ and $Q(x,y)$ are equal. This strongly suggests a linear combination of $g(xyF(x,y)), g(xyQ(x,y)), g\in G$ has orbit-sum $0$, which is an algebraic criteria in quarter-plane lattice walk model \cite{bousquet2005walks}. In \cite{bousquet2023walks}, the author shows that if we consider $F(x,y)$ and $Q(x,y)$ for simple lattice walk $\{\leftarrow,\uparrow,\rightarrow,\downarrow\}$ in three quadrants, there is an algebraic function $A(x,y)$ satisfying
\begin{align}
        xy A(x,y)-y\bar{x}A(\bar{x},y)+\bar{xy}A\left(\bar{x},\bar{y}\right)-x\bar{y}A\left(x,\bar{y}\right)=0,
\end{align}
and 
\begin{align}
    F(x,y)=A(x,y)+\frac{1}{3} (Q(x,y)-\bar{x}^2 Q(\bar{x},y)-\bar{y}^2 Q\left(x,\bar{y}\right)).
\end{align}

The reason we choose the model $\{\nearrow,\nwarrow,\downarrow\}$ is due to the classification in \cite{bousquet2023walks}. The authors classified the $23$ models associated with finite group into two cases,
\begin{itemize}
    \item $7+4$ models with a monomial group (for every $g\in G$, the pair $g(x,y)$ consists of two Laurent monomials in $x$ and $y$). An example is the simple lattice walk. For walks restricted in one quadrant, seven of these models can be solved by reflection principles \cite{Gessel1992} and this is why they are called Weyl models. Four of them can be deformed to walks in a Weyl chamber and the corresponding quarter-plane models of these four have algebraic generating functions.
    \item $12$ non-monomial models which does not satisfy the condition. An example is the group of model with allow steps $\{\nearrow,\nwarrow,\downarrow\}$ shown in \eqref{group for our model}.
\end{itemize}

For the $4$ algebraic models, they automatically have orbit sum $0$ and we can choose $F(x,y)=A(x,y)$. In \cite{bousquet2023walks}, the authors proved for the King walk $(\{\leftarrow,\nwarrow,\uparrow,\nearrow,\rightarrow,\searrow,\downarrow,\swarrow\})$, and conjectured (\textbf{Conjecture 3.2}) for any of the seven Weyl models in $d$-dimension, one can find $A(x,y)$ in a general form,
\begin{align}
    A(x,y)=F(x,y)-\frac{\bar{x}\bar{y}}{2d-1}\left(\frac{OS(xy)}{K(x,y)}-\epsilon_\omega \omega(xyQ(x,y))\right).
\end{align}
$\omega$ is the group element $\omega=\phi\psi\phi\dots$ with length $l(\omega)=d$ and $\epsilon_\omega$ is the sign in the orbit sum. $A(x,y)$ is algebraic, $F(x,y)$ is D-finite. Further the boundary $Hn(\bar{x})$ satisfies the following equation,
\begin{align}
    Hn(\bar{x})=A_{-,0}(\bar{x})+\frac{(-1)^{d-1}}{2d-1}\Big\{\begin{array}{cc}
       \bar{x}^d Q(\bar{x},0)  & \text{if } d=2,4\\
       \bar{x}^d Q(0,x)  &  \text{if } d=3.
    \end{array}\label{condition}
\end{align}
where $A_{-,0}(\bar{x})=[x^<]A(x,0)$ is also algebraic.

The author also noted that for the $12$ non-monomial models, one cannot find such $A(x,y)$. They further noted that constructing such equations as \eqref{condition} requires specializations of \( F(x,y) \).

The model $\{\nearrow,\nwarrow,\downarrow\}$ is the simplest non-monomial case. It also does not satisfy $x/y$ reflection symmetry and cannot be solved via the analytic methods in \cite{raschel2020counting}. It only has a vertical reflection symmetry. 

In this section, we show that the algebraic property still exists. We can find two independent algebraic functions. Each satisfies a polynomial equation with one catalytic variable and is solvable by the general strategy introduced in \cite{bousquet2006polynomial}. A deeper theoretical analysis of this approach is provided in \cref{universal}.

\subsection{Representations of the Generating Functions}
We construct a combinatorial Riemann boundary value problem (cRBVP) by the obstinate kernel method \cite{xu2022interacting}. We draw some inspirations from the representation theory of Lie algebra. Every finite-dimensional representation admits a weight decomposition, and the representation is generated from the highest weight vectors. We have a similar situation in lattice walk problems. For example if we consider the generating functions of simple walks in the quarter-plane and denote the generating function of paths on line $j=k$ as $Q_k(x,0)$, then 
\begin{align}
    Q_1(x,0)=\frac{1}{t}Q(x,0)-\frac{1}{t}+\bar{x}Q(0,0)-(x+\bar{x})Q(x,0),
\end{align}
and
\begin{align}
    Q_{k+1}(x,0)=\frac{1}{t}Q_k(x,0)-Q_{k-1}(x,0)+\bar{x}Q_{0,k}-(x+\bar{x})Q_k(x,0).
\end{align}
$Q(x,0)$ acts as the highest weight `vector' and all $Q_k(x,0)$ are generated from $Q(x,0)$.

For our model, we can generate $Hn_k(\bar{x})$\footnote{$Hn_k(\bar{x})$ refers to horizontal negative $j=k$ line.} for $k\geq 0$ from $Hn(\bar{x})$. However, for the right half plane, to generate all the $Hp_k(x)$for $k\in \mathbb{Z}$, we need two $Hp_i(x)$ as generators. So, the matrix cRBVP shall involve three independent unknown functions. Let us choose $Hn(\bar{x}),Hp(x)$ and $Hp_{-1}(x)$ as three generators.

\subsection{Matrix cRBVP for Three-Quadrant Walks}
Now, let us start constructing the cRBVP. The cRBVP is an automorphism relation between $Hn(x)$, $Hp(x)$, $Hp_{-1}(x)$ and $Hn(\bar{x})$, $Hp(\bar{x})$, $Hp_{-1}(\bar{x})$. This suggests that we shall consider the generating function of paths ending in the upper half-plane and in the fourth quadrant which have these generators as boundaries. 

Denote the generating function of the walks ending in the upper half plane (including the $j=0$ axis) as $U(x,y)$. It satisfies,
\begin{align}
\begin{split}
&K(x,y)U(x,y)=1+t((x+\bar{x})y+\bar{y})U(x,y)\\
&+t(x+\bar{x})Hp_{-1}(x)-t\bar{x}F_{0,-1}-t\bar{y}HN(\bar{x})-t\bar{y}Hp(x)+t p \bar{x}F_{0,-1}\label{upper half plane}.
\end{split}
\end{align}

\eqref{upper half plane} can be rewritten in a kernel form,
\begin{align}
(1-t((x+\bar{x})y+\bar{y}))U(x,y)=1+t(x+\bar{x})HP_{-1}(x)-t(1-p)\bar{x}F_{0,-1}-t\bar{y}HN(\bar{x})-t\bar{y}HP(x)\label{upper half plane kernel form}.
\end{align}

The kernel $K(x,y)=(1-t((x+\bar{x})y+\bar{y}))$ has two roots as a function of $y$, namely,
\begin{align}
    \begin{split}
        &Y_0(x)=\frac{x-\sqrt{-4 t^2 x^3-4 t^2 x+x^2}}{2 \left(t x^2+t\right)}=t+t^3 \left(x+\frac{1}{x}\right)+\frac{2 t^5 \left(x^2+1\right)^2}{x^2}+\frac{5 t^7 \left(x^2+1\right)^3}{x^3}+O\left(t^9\right),\\
        &Y_1(x)=\frac{x+\sqrt{-4 t^2 x^3-4 t^2 x+x^2}}{2 \left(t x^2+t\right)}=\frac{x}{t \left(x^2+1\right)}-t-\frac{t^3 \left(x^2+1\right)}{x}-\frac{2 t^5 \left(x^2+1\right)^2}{x^2}+O\left(t^7\right).
    \end{split}
\end{align}
and 
\begin{align}
Y_0(x)Y_1(x)=\frac{1}{x+\bar{x}}, \qquad Y_0(x)+Y_1(x)=\frac{1}{t(\bar{x}+x)}.
\end{align}
Both roots are invariant under $x\to \bar{x}$. $Y_0(x)$ is analytic at $t=0$ and can be expanded as a formal series in $t$ while $Y_1(x)$ has a pole at $t=0$. Its expansion around $0<t<\epsilon$ for some small value $\epsilon$ contains a $1/t$ term.

$U(x,y)$ is a formal series in $\mathbb{C}[y,x,1/x][[t]]$, so we can substitute $y=Y_0(x)$ into it. $U(x,Y_0(x))$ is a well-defined formal pwoer series (or convergent series with finite $x$ and $t<\epsilon$ from an analytic point of view) and $K(x,Y_0(x))=0$. This eliminates $U(x,Y_0(x))$. Applying the transform $x\to 1/x$ and we get two automorphism relations,
\begin{align}
        &-t\bar{x}(1-p)F_{0,-1}-\frac{t Hn\left(\bar{x}\right)}{Y_0}+t \left(x+\bar{x}\right) Hp_{-1}(x)-\frac{t Hp(x)}{Y_0}+1=0\label{upper half plane 1},\\
        &-t x(1-p)F_{0,-1}-\frac{t Hn(x)}{Y_0}+t \left(x+\bar{x}\right) Hp_{-1}\left(\bar{x}\right)-\frac{t Hp\left(\bar{x}\right)}{Y_0}+1=0\label{upper half plane 2}.
\end{align}

$Y_0(x)$ is abbreviated as $Y_0$ in later calculations. 

The third automorphism relation comes from the fourth quadrant and is obtained from a similar calculation in quarter-plane models. Denote the generating functions of the paths ending in the fourth quadrant as $V(x,y)$. It satisfies the following equation,
\begin{align}
    K(x,y)V(x,y)=t(1-p)\bar{x}F_{0,-1}-t \left(x+\bar{x}\right) Hp_{-1}(x)+t\bar{y} Hp(x)-t \bar{y} Vn\left(\bar{y}\right).
\end{align}
$V(x,y)$ is in $\mathbb{C}[x,\bar{y}][[t]]$. We should substitute $y=Y_1(x)$ into the equation, since $1/Y_1=(x+\bar{x})Y_0\in \mathbb{C}(x)[[t]]$. $Vn(x,Y_0)$ is a well-defined substitution. We have,
\begin{align}
    t(1-p)\bar{x}F_{0,-1}+t \left(x+\bar{x}\right) Y_0 Hp(x)-t \left(x+\bar{x}\right) Hp_{-1}(x)-\frac{t Vn\left(\left(x+\bar{x}\right) Y_0\right)}{x\left(x+\bar{x}\right) Y_0}=0\label{fourth quadrant}.
\end{align}
Applying the transformation $x\to 1/x$ to \eqref{fourth quadrant} we get,
\begin{align}
    t(1-p)x F_{0,-1}+t \left(x+\bar{x}\right) Y_0 Hp\left(\bar{x}\right)-t \left(x+\bar{x}\right) Hp_{-1}\left(\bar{x}\right)-\frac{t x Vn\left(\left(x+\bar{x}\right) Y_0\right)}{\left(x+\bar{x}\right) Y_0}=0 \label{fourth quadrant 2}.
\end{align}
By a linear combination of \eqref{fourth quadrant} and \eqref{fourth quadrant 2}, we eliminate $Vn((x+\bar{x})Y_0(x))$ and get the third automorphism relation,
\begin{align}
    \bar{x}Y_0 Hp\left(\bar{x}\right)- xY_0Hp(x)+x Hp_{-1}(x)-\bar{x}Hp_{-1}\left(\bar{x}\right)=0\label{fourth quadrant 3}.
\end{align}

\subsection{The Null-Space}
\eqref{upper half plane 1}, \eqref{upper half plane 2}, \eqref{fourth quadrant 3} together gives a matrix cRBVP in $Hn(x),Hp(x),Hp_{-1}(x)$,
\begin{align}
    \left(\begin{array}{c}
    Hp(\bar{x})\\
    Hp_{-1}(\bar{x})\\
    Hn(\bar{x})
    \end{array}\right)=M
        \left(\begin{array}{c}
    Hp(x)\\
    Hp_{-1}(x)\\
    Hn(x)
    \end{array}\right)
    +
    \left(\begin{array}{c}
    -\frac{x Y_0 ((p-1) t x F_{0,-1}+1)}{t \left(x^2 Y_0^2-x+Y_0^2\right)}\\
    -\frac{x Y_0^2 ((p-1)t x F_{0,-1}+1)}{t \left(x^2 Y_0^2-x+Y_0^2\right)}\\
    \frac{Y_0 ((p-1) t F_{0,-1}+x)}{t x} 
    \end{array}\right)\label{RBVP for model 3},
\end{align}
where
\begin{align}
    M=\left(
\begin{array}{ccc}
 \frac{x^2 \left(x^2+1\right) Y_0^2}{x^2 Y_0^2-x+Y_0^2} & -\frac{x^2 \left(x^2+1\right) Y_0}{x^2 Y_0^2-x+Y_0^2} & \frac{x}{x^2 Y_0^2-x+Y_0^2} \\
 \frac{x^3 Y_0}{x^2 Y_0^2-x+Y_0^2} & -\frac{x^3}{x^2 Y_0^2-x+Y_0^2} & \frac{x Y_0}{x^2 Y_0^2-x+Y_0^2} \\
 -1 & \frac{\left(x^2+1\right) Y_0}{x} & 0 \\
\end{array}
\right).
\end{align}
Since $K(x,Y_0)=0$, we can simplify $M$ such that each entry in $M$ is linear in $Y_0$,
\begin{align}
    M=\left(
\begin{array}{ccc}
 \frac{t x^2 \left(x^2+1\right) \left(2 t-Y_0\right)}{4 t^2 x^2+4 t^2-x} & \frac{t x \left(x^2+1\right) \left(2 t x^2 Y_0+2 t Y_0-x\right)}{4 t^2 x^2+4 t^2-x} & -\frac{2 t^2 x^2+2 t^2+t x^2 Y_0+t Y_0-x}{4 t^2 x^2+4 t^2-x} \\
 -\frac{t x^2 \left(2 t x^2 Y_0+2 t Y_0-x\right)}{4 t^2 x^2+4 t^2-x} & \frac{x^2 \left(2 t^2 x^2+2 t^2+t x^2 Y_0+t Y_0-x\right)}{4 t^2 x^2+4 t^2-x} & -\frac{t \left(2 t x^2 Y_0+2 t Y_0-x\right)}{4 t^2 x^2+4 t^2-x} \\
 -1 & \frac{\left(x^2+1\right) Y_0}{x} & 0 \\
\end{array}
\right).
\end{align}
Substitute $Y_0(x)=\frac{1-\sqrt{\Delta}}{2 t (x+\bar{x})}$ in and write $M=P_0(x)+P_1(x)\sqrt\Delta$. $\sqrt{\Delta}=\sqrt{-4 t^2(x+\bar{x})+1}$.
\begin{align}
    M=\left(
\begin{array}{ccc}
 \frac{x^2}{2} & 0 & -\frac{1}{2} \\
 0 & \frac{x^2}{2} & 0 \\
 -1 & \frac{1}{2 t} & 0 \\
\end{array}
\right)+\sqrt{\Delta}
\left(
\begin{array}{ccc}
 \frac{x^3}{2 \left(4 t^2 x^2+4 t^2-x\right)} & -\frac{t x^2 \left(x^2+1\right)}{4 t^2 x^2+4 t^2-x} & \frac{x}{2 \left(4 t^2 x^2+4 t^2-x\right)} \\
 \frac{t x^3}{4 t^2 x^2+4 t^2-x} & -\frac{x^3}{2 \left(4 t^2 x^2+4 t^2-x\right)} & \frac{t x}{4 t^2 x^2+4 t^2-x} \\
 0 & -\frac{1}{2 t} & 0 \\
\end{array}
\right).
\end{align}
This is an equation in the form \eqref{independent 1}. Further, we can check that $Det|P_1(x)|=0$ and $Det|P_0(x)|\neq 0$. Thus, there is an left null-vector $v(\bar{x})$ such that $v(\bar{x})P_1(x)=0$. By some simple calculation, we have,
\begin{align}
    v(\bar{x})=\left(-2\bar{x}^2,\frac{\bar{x}^2}{t},1\right).
\end{align}
Multiply $v(\bar{x})$ on the left to \eqref{RBVP for model 3}. Since $v(\bar{x})P_1(x)=(0,0,0)$, we eliminate $\sqrt{\Delta}$ in the coefficients of $Hp(x),Hp_{-1}(x)$ and $Hn(x)$. The equation then reads,
\begin{align}
    -tx Hn\left(\bar{x}\right)+t\bar{x}Hn(x)+x Hp_{-1}(x)-\bar{x}Hp_{-1}\left(\bar{x}\right)+2t\bar{x}Hp\left(\bar{x}\right)-2txHp(x)+(x-\bar{x})Y_0=0\label{Matrix RBVP Model3 final}.
\end{align}
This is a linear equation of $Hp_{-1}(\bar{x}),Hp(\bar{x}),Hn(\bar{x})$ and $Hp_{-1}(x),Hp(x),Hn(x)$ with rational coefficients. It is suitable for taking $[x^>]$ and $[x^<]$ terms. We call an equation with this property a separable equation. Denote $PR(x)=[x^>](x-\bar{x})Y_0$ and $NR(x)=[x^<](x-\bar{x})Y_0$, \eqref{Matrix RBVP Model3 final} gives a linear relation among three unknown functions,
\begin{align}
    \begin{split}
        &t\bar{x}(Hn(x)-x F_{-1,0})-2 t x Hp(x)+x Hp_{-1}(x)+PR(x)=0\\
       & -t x (Hn\left(\bar{x}\right)-\bar{x}F_{-1,0})+2 t\bar{x} Hp\left(\bar{x}\right)-\bar{x}Hp_{-1}\left(\bar{x}\right)+NR\left(\bar{x}\right)=0\label{pos and neg}.
    \end{split}
\end{align}
\begin{remark}\label{orbit sum and null space}
    There is another way to find \eqref{Matrix RBVP Model3 final}. Take the $[y^1]$ degree term of the full orbit sum \eqref{full orbit sum Model 3},
\begin{align}
\begin{split}
    &x(Hn(x)+Hp(x))-\bar{x}(Hn(\bar{x})+Hp(\bar{x}))\\
    &+\bar{x}(x+\bar{x})Hp_{-2}(\bar{x})-x(x+\bar{x})Hp_{-2}(x)=-\frac{(x-1) (x+1) Y_0}{t x}\label{orbit sum and null space 111}.
\end{split}
\end{align}  
    Then by the recurrent relation between $Hp_{-2}(x,0),Hp_{-1}(x,0),Hp(x,0)$, we get exactly \eqref{Matrix RBVP Model3 final}. This shows that the full orbit sum condition is equivalent to the rank condition of $P_1(x)$ in the matrix cRBVP. We have more discussions on this in \cref{universal}.
\end{remark}
\subsection{Another Separable Equation}
\eqref{pos and neg} is not enough to solve the problem. We need to find at least two separable equations. To find another separable equation, we substitute \eqref{pos and neg} into \eqref{matrix RBVP} and eliminate $Hp(x),Hp(\bar{x})$, we get the following two equations,
\begin{align}
\begin{split}
&\sqrt{-4 t^2 x-4 t^2 \bar{x}+1} \left(-\frac{(p-1) F_{0,-1}}{x+\bar{x}}-xHp_{-1}(x)-\frac{x^2}{t \left(x^2+1\right)}\right)\\
&+\frac{(p-1)  F_{0,-1}}{x+\bar{x}}+t F_{-1,0}-2 t x Hn\left(\bar{x}\right)-t\bar{x}Hn(x)-PR(x)+\frac{x^2}{t \left(x^2+1\right)}=0,
\end{split}
\end{align}
\begin{align}
\begin{split}
&\sqrt{-4 t^2 x-4 t^2\bar{x}+1} \left(-\frac{2 (p-1) F_{0,-1}}{x+\bar{x}}-2\bar{x} Hp_{-1}\left(\bar{x}\right)-\frac{2}{t \left(x^2+1\right)}\right)\\
&\frac{2 (p-1)F_{0,-1}}{x+\bar{x}}+2 t F_{-1,0}-2 t xHn\left(\bar{x}\right)-4\bar{x} t Hn(x)+2NR(\bar{x})+\frac{2}{t \left(x^2+1\right)}=0
\end{split}
\end{align}
Here we use the definition of $NR(\bar{x})+PR(x)=(x-\bar{x})Y_0$ to move $NR(\bar{x})$ and $PR(x)$ outside the coefficients of $\sqrt{-4 t^2 (x+ \bar{x})+1}$. The difference of the above two equations gives,
\begin{align}
\begin{split}
    &\sqrt{-4 t^2 x-4 t^2\bar{x}+1} \left(\frac{(1-p)tF_{0,-1}}{ \left(x+\bar{x}\right)}-2t\bar{x} Hp_{-1}\left(\bar{x}\right)+tx Hp_{-1}(x)-\frac{1}{\left(x^2+1\right)}\right)\\
    &-\frac{x(1-p) F_{0,-1}}{t \left(x^2+1\right)}+t^2F_{-1,0}-3 t^2\bar{x}Hn(x)-t  PR(x)+\frac{1}{ \left(x^2+1\right)}=0\label{qudratic method 1}.
\end{split}
\end{align}
\eqref{qudratic method 1} has a very special form. If we denote
\begin{align}
\begin{split}
    &HL=-\frac{x(1-p) F_{0,-1}}{t \left(x^2+1\right)}+t^2F_{-1,0}-3 t^2\bar{x}Hn(x)-t  PR(x)+\frac{1}{ \left(x^2+1\right)},\\
    &HR=\left(\frac{(1-p)tF_{0,-1}}{ \left(x+\bar{x}\right)}-2t\bar{x} Hp_{-1}\left(\bar{x}\right)+tx Hp_{-1}(x)-\frac{1}{\left(x^2+1\right)}\right).
\end{split}
\end{align}
$HL$ only contains one unknown function in $x$, which is $Hn(x)$. $HR$ contains $\bar{x}Hp(\bar{x})$ and $xHp(x)$ and the equation reads,
\begin{align}
    HL=-\left(\sqrt{-4 t^2 x-4 t^2 \bar{x}+1}\right)HR.
\end{align}
Now, we can square $HL$ and eliminate the square-root in the equation. Further notice that if we let $p=1$, $HL$ and $HR$ will be simplified. Without loss of generality, we consider the case $p=1$. After squaring \eqref{qudratic method 1}, the equation reads,
\begin{align}
    \begin{split}
        &PR(x) \left(-2 tF_{-1,0}-\frac{2}{t \left(x^2+1\right)}\right)+\frac{2 F_{-1,0}}{ \left(x^2+1\right)}+t^2F_{-1,0}^2+PR(x)^2+\frac{4}{x \left(x^2+1\right)}\\
        &+9 t^2\bar{x}^2Hn(x)^2+Hn(x) \left(-6t^2\bar{x} F_{-1,0}+6t\bar{x}PR(x)-\frac{6}{x \left(x^2+1\right)}\right)\\
        &+\frac{4 Hp_{-1}\left(\bar{x}\right)^2 \left(4 t^2 x^2+4 t^2-x\right)}{ x^3}+\frac{4 Hp_{-1}\left(\bar{x}\right) \left(4 t^2 x^2+4 t^2-x\right)}{t x^2 \left(x^2+1\right)}\\
        &x Hp_{-1}(x)^2 \left(4 t^2 x^2+4 t^2-x\right)-\frac{2 Hp_{-1}(x) \left(4 t^2 x^2+4 t^2-x\right)}{t \left(x^2+1\right)}\\
        &-4 Hp_{-1}\left(\bar{x}\right) Hp_{-1}(x) \left(4 t^2x+4 t^2\bar{x}-1\right)=0\label{qudratic method 2}.
    \end{split}
\end{align}
We have two objectives,
\begin{enumerate}
    \item We want to separate the functions $Hp_{-1}(x)$ and $Hn(x)$.
    \item We want to separate $Hp_{-1}(x)$ and $Hp_{-1}(\bar{x})$.
\end{enumerate}
The main term preventing us from taking $[x^>]$ and $[x^<]$ is $4 Hp_{-1}\left(\bar{x}\right) Hp_{-1}(x) \left(4 t^2 x+4 t^2\bar{x}-1\right)$, because it is a product of a formal series in $x$ and a formal series in $\bar{x}$. However, it is symmetric under the transform $x\to 1/x$. Following the idea of \cite{bousquet2016square}, we first separate $Hp_{-1}(x)$ and $Hn(x)$. 

Without loss of generality, let us consider the following equation,
\begin{align}
    P(x)-A(x,\bar{x})+Q(\bar{x})=R(x)\label{qudratic method 2 ab}
\end{align}
where $P(x),R(x)$ are the sums of formal series in $x$ and finite number of $[x^<]$ monomials. $Q(\bar{x})$ are sums of formal series in $\bar{x}$ and finite number of $[x^>]$ monomials. $A(x,\bar{x})$ is symmetric under the transformation $x\to 1/x$.

To solve \eqref{qudratic method 2 ab}, first take the $[x^<]$ terms of it,
\begin{align}
    \left([x^<](P(x)-R(x)\right)+Q(\bar{x})-[x^>]Q(\bar{x})-[x^<]A(x,\bar{x})=0.\label{qudratic method 31}
\end{align}
The formal series part of $R(x),P(x)$ are eliminated by taking $[x^<]$ terms. Then apply the transformation $x\to 1/x$ to \eqref{qudratic method 31}, we have,
\begin{align}
        \left([x^>](P(\bar{x})-R(\bar{x})\right)+Q(x)-[x^<]Q(x)-[x^>]A(\bar{x},x)=0\label{qudratic method 32}.
\end{align}
The $[x^0]$ terms of \eqref{qudratic method 2 ab} shows,
\begin{align}
    [x^0]P(x)-[x^0]A(x,\bar{x})+[x^0]Q(\bar{x})=[x^0]R(x)\label{qudratic method 33}.
\end{align}
Notice the trivial equality,
\begin{align}
    [x^<]A(x,\bar{x})+[x^0]A(x,\bar{x})+[x^>]A(x,\bar{x})=A(x,\bar{x}).
\end{align}
$\eqref{qudratic method 31}+\eqref{qudratic method 33}+\eqref{qudratic method 32}$ gives,
\begin{align}
    Q(x)-A(x,\bar{x})+Q(\bar{x})=\text{sum of monomials in $x$}.\label{qudratic method reflect}
\end{align}
In \eqref{qudratic method reflect} we eliminate $R(x)$ and $P(x)$ in the equation and get an equation of $Q(x)$. In our case \eqref{qudratic method 2}, $A(x,\bar{x})$ contains $Hp_{-1}(x)Hp_{-1}(\bar{x})$, $P(x)$ contains $Hp_{-1}(x)$, $R(x)$ contains $Hn(x)$ and $Q(\bar{x})$ contains $Hp(\bar{x})$. So we eliminate $Hn(x)$ from \eqref{qudratic method 2} and find an equation of $Hp_{-1}(x),Hp_{-1}(\bar{x})$.

However, in our case, $P(x),R(x)$ are not sums of formal series and finite number of $[x^<]$ monomials. It is a formal series in $x$ with rational coefficients. We apply the following theorem in  \cite{bousquet2021more}.
\begin{lemma}(None-negative part at a pole).\label{lemma 3}
    Let $F(x)\in \mathbb{C}[x][[t]]$ and $\rho \in \mathbb{C}$. Then,
    \begin{align}
        &[x^\geq]\frac{F(\bar{x})}{1-\rho x}=\frac{ F(\rho)}{1-\rho x},\\
        &[x^0]\frac{F(\bar{x})}{1-\rho x}=F(\rho),\\
        &[x^\geq]\frac{F(\bar{x})}{(1-\rho x)^2}=\frac{ F(\rho)}{(1-\rho x)^2}+\frac{\rho F(\rho)}{1-\rho x}.
    \end{align}
\end{lemma}
\begin{proof}
    Expanding $\frac{1}{1-\rho x}$ as a formal series of $\rho x$, we have,
    \begin{align}
        [x^\geq]\frac{\bar{x}^k}{1-\rho x}=\bar{x}^k\sum_{n\geq k}\rho^n x^n=\frac{\rho^k}{1-\rho x},
    \end{align}
    and
    \begin{align}
    [x^\geq]\frac{\bar{x}^k}{(1-\rho x)^2}=\bar{x}^k\sum_{n\geq k}(n+1)\rho^n x^n=\rho^k\sum_{n\geq k}(n+k+1)\rho^n x^n=\frac{k\rho^k}{1-\rho x}+\frac{\rho^k}{(1-\rho x)^2}.
    \end{align}
These two equations are valid for any monomials and thus hold for $F(\bar{x})$.
\end{proof}
\begin{remark}
    To apply \cref{lemma 3}, we extend the condition from $\rho\in\mathbb{C}$ to $\rho\in\mathbb{C}[[t]]$. We need to take care of the `value' of each pole carefully, such that all series are compatible with each other. This is because we already have some conditions on the convergent domain.  Analytically speaking, the series expansion of $Y_0(x)$ only holds in the annulus away from the branch cut of $\sqrt{\Delta}$. All Laurent series shall be expanded in the same annulus.
\end{remark}
Let us apply \cref{lemma 3} to \eqref{qudratic method 2}. We have a factor $1+x^2$ in the denominator. $Hp_{-1}(x),Hn(x)$ are both defined as formal series in $\mathbb{C}(x)[[t]]$. If we consider them as convergent series on $x$-plane for fixed $0<t<\epsilon$,
\begin{align}
    Hp_{-1}(x)=\sum_{i=0}^\infty \sum_{k=i+1}^{\infty}f_{i,-1,k}t^kx^i.
\end{align}
The convergent domain is $|(\sum_{k=i+1}^{\infty}f_{i,-1,k}t^k)x^i|^{1/i}<1$, or briefly, $|x|<A/t$ for some constant $A$. For $Hp_{-1}(\bar{x})$, the convergent domain is $|x|>t/A$. Thus, if we consider Laurent expansion in the annulus $t/A<|x|<A/t$, \eqref{qudratic method 2} coincide with the formal series definition. For small $t$ the circle $|x|=1$ is inside this annulus. We can choose the series expansion of $\frac{1}{1+x^2}$ either as $\frac{1}{1+x^2}=\sum_{i=0}(-x^2)^i$ or $\frac{1}{1+x^2}=\sum_{i=1}(-\bar{x})^i$. $1+x^2$ appears as the denominator of $Hp_{-1}(x),Hp_{-1}(\bar{x})$ and $PR(x)$. For simplicity, we choose $\frac{1}{1+x^2}=\sum_{i=0}(-x^2)^i$ and the convergent domain becomes $t/A<|x|<1$ in the future. The series expansion of $\frac{Hp_{-1}(x)}{1+x^2}$ and $\frac{PR(x)}{1+x^2}$ still belong to $\mathbb{C}(x)[[t]]$ and by \cref{lemma 3}, 
\begin{align}
\begin{split}
&[x^>]\frac{4 Hp_{-1}\left(\bar{x}\right) \left(4 t^2 x^2+4 t^2-x\right)}{t x^2 \left(x^2+1\right)}=[x^>]\frac{4 Hp_{-1}\left(\bar{x}\right) \left(4 t^2 x^2+4 t^2-x\right)}{t x^2}\frac{1}{2}\left(\frac{1}{1-ix}+\frac{1}{1+ix}\right)\\
&=\frac{2 i Hp_{-1}(-i)}{t (1+i x)}-\frac{2 i Hp_{-1}(i)}{t (1-i x)}.
\end{split}
\end{align}
\eqref{qudratic method reflect} for our model reads,
\begin{align}
    \begin{split}
        &\frac{Hp_{-1}\left(\bar{x}\right)}{t x \left(x^2+1\right)}+\frac{x^3 Hp_{-1}(x)}{t \left(x^2+1\right)}+x^2 Hp_{-1}(x){}^2+\bar{x}^2Hp_{-1}\left(\bar{x}\right)^2-HP_{-1}\left(\bar{x}\right)Hp_{-1}(x)\\
        &-\frac{Hp_{-1}(i) x^2}{t (x-i) (x+i) \left(4 t^2 x^2+4 t^2-x\right)}-\frac{HP_{-1}(-i) x^2}{t (x-i) (x+i) \left(4 t^2 x^2+4 t^2-x\right)}\\
        &+\frac{t^2 x F_{-1,0}^2-x F_{-1,0}+x^2+1}{4 t^2 x^2+4 t^2-x}-\frac{2 t x F_{0,-1}}{4 t^2 x^2+4 t^2-x}=0.\label{poly 1}
    \end{split}
\end{align}
If we denote $A(x)=x\left(Hp_{-1}(x)+\frac{1+2x^2}{3t(1+x^2)}\right)$, \eqref{poly 1} reads,
\begin{align}
\begin{split}
    &A(\bar{x})^2-A(x)A(\bar{x})+A(x)^2\\
    &+\frac{t^2 x F_{-1,0}^2}{4 t^2 x^2+4 t^2-x}-\frac{x F_{-1,0}}{4 t^2 x^2+4 t^2-x}-\frac{2 t x F_{0,-1}}{4 t^2 x^2+4 t^2-x}\\
    &-\frac{Hp_{-1}(-i) x^2}{t (x^2+1) \left(4 t^2 x^2+4 t^2-x\right)}-\frac{Hp_{-1}(i) x^2}{t (x^2+1) \left(4 t^2 x^2+4 t^2-x\right)}\\
    &-\frac{t^2 x^6-t^2 x^4-t^2 x^2+t^2-x^5-x^3-x}{3 t^2 \left(x^2+1\right)^2 \left(4 t^2 x^2+4 t^2-x\right)}=0\label{poly 2}
\end{split}
\end{align}
In \eqref{poly 2}, the terms of $A(x)$ form an expression of cyclotomic polynomial  $a^2-ab+b^2$. Thus multiplying \eqref{poly 2} by $A(\bar{x})+A(x)$ and $\left(4 t^2 x+4 t^2\bar{x}-1\right)$ to clear the denominator, we have
\begin{align}
    \begin{split}
        &\bar{x}^3Hp_{-1}\left(\bar{x}\right)^3 \left(4 t^2 x+4 t^2\bar{x}-1\right)+\frac{\left(x^2+2\right) Hp_{-1}\left(\bar{x}\right){}^2 \left(4 t^2 x^2+4 t^2-x\right)}{t x^3 \left(x^2+1\right)}\\
        &+x^3 Hp_{-1}(x)^3 \left(4 t^2 x+4 t^2\bar{x}-1\right)+\frac{x\left(2 x^2+1\right) Hp_{-1}(x){}^2 \left(4 t^2 x^2+4 t^2-x\right)}{t \left(x^2+1\right)}\\
        &+Hp_{-1}\left(\bar{x}\right) \left(\bar{x}(t^2 F_{-1,0}^2+F_{-1,0})-2 t\bar{x} F_{0,-1}-\frac{Hp_{-1}(-i)+Hp_{-1}(i)}{t (x-i) (x+i)}+\frac{t^2 x^4+6 t^2 x^2+5 t^2-x}{t^2 x^2 \left(x^2+1\right)}\right)\\
        &+Hp_{-1}(x) \left( x (t^2F_{-1,0}^2+F_{-1,0})-2t x F_{0,-1}-\frac{(Hp_{-1}(-i)+Hp_{-1}(i)) x^2}{t (x-i) (x+i)}+\frac{5 t^2 x^4+6 t^2 x^2+t^2-x^3}{t^2 \left(x^2+1\right)}\right)\\
        &+t F_{-1,0}^2-\frac{F_{-1,0}}{t}-2 F_{0,-1}-\frac{(Hp_{-1}(-i)+Hp_{-1}(i)) x}{t^2 (x-i) (x+i)}+\frac{x^2+1}{t x}=0\label{poly 3}.
    \end{split}
\end{align}
In \eqref{poly 3}, the unknown functions in $\mathbb{C}[x][[t]]$ and unknown functions in $\mathbb{C}[\bar{x}][[t]]$ are separated. Thus we can take the $[x^>]$ and $[x^<]$ part of this equation. We still need to apply \cref{lemma 3} because of the term $\frac{Hp_{-1}(\bar{x})^2}{x^2+1}$. Further, we find $[x^0]$ term of \eqref{poly 3} reads
\begin{align}
    \frac{2 t^3 F_{-1,0}^2-2 t F_{-1,0}-2 Hp_{-1}(-i) Hp_{-1}(i) t+i Hp_{-1}(-i)-i Hp_{-1}(i)}{2 t^2}=0.
\end{align}
This solves
\begin{align}
    Hp_{-1}(-i)= \frac{2 t^3 F_{-1,0}^2-2 t F_{-1,0}-i Hp_{-1}(i)}{2 Hp_{-1}(i) t-i}.
\end{align}
We can eliminate $Hp_{-1}(-i)$ in the equation and reduce the number of unknowns. Actually, a more convenient substitution is to consider,
\begin{align}
    F_{-1,0}^2=\frac{2 t F_{-1,0}+2 Hp_{-1}(-i) Hp_{-1}(i) t-i Hp_{-1}(-i)+i Hp_{-1}(i)}{2 t^3}.\label{Fm1}
\end{align}
This will fortunately eliminate $F_{-1,0}$ in the $[x^>]$ part of \eqref{poly 3}.

After substitution, thte $[x^>]$ part of \eqref{poly 3} reads,
\begin{align}
    \begin{split}
        &t^2 x^2 \left(x^2+1\right) Hp_{-1}(x){}^3 \left(4 t^2 x^2+4 t^2-x\right)+t x \left(2 x^2+1\right) Hp_{-1}(x){}^2 \left(4 t^2 x^2+4 t^2-x\right)\\
        &+Hp_{-1}(x) \Big(-2 t^3x( x^2+1) Q_{0,-1}+Hp_{-1}(-i)Hp_{-1}(i) t^2 \left(x^2+1\right) x\\
        &-\frac{1}{2}Hp_{-1}(-i) it \left(x^2-2 i x+1\right) x+\frac{1}{2}Hp_{-1}(i) it \left(x^2+2 i x+1\right) x+5 t^2 x^4+6 t^2 x^2+t^2-x^3\Big)\\
        &-t^2 \left(x^2+1\right) F_{0,-1}+Hp_{-1}(-i) Hp_{-1}(i) t x^2+t x\left(x^2+1\right)\\
        &-\frac{1}{2} Hp_{-1}(-i) i(x-i) x+\frac{1}{2} Hp_{-1}(i) i(x+i) x=0\label{poly final}.
    \end{split}
\end{align}
\subsection{Existence and Uniqueness of the Solution}
\eqref{poly final} is a polynomial equation with one catalytic variable
\begin{align}
    P(Hp_{-1}(x),F_{0,-1},Hp_{-1}(i),Hp_{-1}(-i);x,t)=0.\label{polynomial equation}
\end{align}
In \cite{bousquet2006polynomial}, the authors introduced a general strategy to solve it. Namely, to solve a polynomial equation in the form,
\begin{equation}
P(Q(x),Q_1,Q_2\dots Q_k,t,x)=0.
\end{equation}
where $Q(x)$ is a formal series of $t$ with coefficients in $\mathbb{C}(x)$ and $Q_i$ are formal series in $t$, one performs the following process,
\begin{enumerate}
\item Differentiate the equation with respect to $x$:
\begin{equation}
Q'(x)\partial_{x_0}P(Q(x),Q_1,Q_2\dots Q_k,t,x)+\partial _x P(Q(x),Q_1,Q_2\dots Q_k,t,x)=0.
\end{equation}
\item Find roots $X$ which satisfy
\begin{equation}
\partial_{x_0}P(Q(X),Q_1,Q_2\dots Q_k,t,X)=0\label{bou 3 non}.
\end{equation}
Then 
\begin{equation}
\partial_x P(Q(X),Q_1,Q_2\dots Q_k,t,X)=0
\end{equation}
automatically holds.
\item If there are $k$ distinct $X_i$ such that \eqref{bou 3 non} holds, then we get $3k$ polynomial equations
\begin{align}
\begin{split}
&P(Q(X_i),Q_1,Q_2\dots Q_k,t,X_i)=0\\
&\partial_{x_0}P(Q(X_i),Q_1,Q_2\dots Q_k,t,X_i)=0\\
&\partial_x P(Q(X_i),Q_1,Q_2\dots Q_k,t,X_i)=0.
\end{split}\label{polyproof matrix}
\end{align}
\end{enumerate}
If these $3k$ equations are independent, the system admits a unique solution for $X_1,\dots X_k$, $Q(X_1)\dots Q(X_k)$, \( Q_1,\dots,Q_k \), thereby determining \( Q(x) \).  

The key challenge of this strategy is to find $k$ distinct $X_i$. However, we failed to find a feasible way to find these $X_i$ \eqref{polynomial equation}. In \cite{bousquet2023walks,bousquet2021more} the authors introduce a guess-and-check procedure to find these roots. We conjecture that this model can also be solved in a similar way.

Instead of finding the exact $X_1,X_2,X_3$, we give a simple proof to show that these $X_i$ exist and \eqref{polynomial equation} is integrable. The proof is based on the following two theorems (\textbf{Theorem 2} and \textbf{Theorem 4} in \cite{bousquet2006polynomial}).
\begin{theorem} (\textbf{Theorem 2} in \cite{bousquet2006polynomial})\label{theorem 2}
 Let $\Phi(x,t) \in \mathbb{K}[x]^{fr}[[t]]$ and $\mathbb{K}$ is an algebraically closed field.
 If $[t^0]\Phi(x,t)=x^k$, then $\Phi(x,t)$ has exactly $k$ roots $X_1,X_2,\dots X_k$ in $\mathbb{K}^{fr}[[t]]$.
\end{theorem}
\begin{theorem}(\textbf{Theorem 4} in \cite{bousquet2006polynomial})\label{theorem 4}
    Let $\mathbb{K}\subset \mathbb{L}$ be a field extension. For $1\leq i\leq n$, let $P_i(x_1,\dots x_n)$ be polynomials in indeterminate $x_1,\dots x_n$, with coefficients in the (small) field $\mathbb{K}$. Assume $F_1\dots F_n$ are $n$ elements of the (big) field $\mathbb{L}$ that satisfy $P_i(F_1,\dots,F_n)=0$ for all $i$. Let $J$ be the Jacobian matrix
    \begin{align}
        J=\left(\frac{\partial P_i}{\partial x_j}(F_1,\dots F_n)\right)_{1\leq i,j\leq n}\label{Jacob}
    \end{align}
    If $Det|J|\neq 0$, then each $F_j$ is algebraic over $\mathbb{K}$.
\end{theorem}
To prove the existence, let us change the variable to simplify the expression. First, $Hp_{-1}(i)$ and $Hp_{1}(i)$ are complex conjugate to each other. We can write
\begin{align}
\begin{split}
    &Hp_{-1}(i)=B_1+B_2 i\\
    &Hp_{-1}(-i)=B_1-B_2 i.
\end{split} 
\end{align}
Then we apply the following substitution,
\begin{align}
    \begin{split}
        &B_3=t (B_1^2+B_2^2)- B_2\\
        &B_4=t^2F_{0,-1}-\frac{1}{2}B_3\\
        &f(x)=Hp_{-1}(x).
    \end{split}
\end{align}
\eqref{poly final} then reads,
\begin{align}
\begin{split}
&B_1 x (2 t f(x)+1)-\frac{1}{2} B_3 \left(x^2-1\right)+B_4 \left(x^2+1\right) (2 t f(x)+1)\\
&-\left(t( x^2 +1) f(x)+x\right) \left(t x^2 f(x)^2 \left(4 t^2 \left(x^2+1\right)-x\right)+x f(x) \left(4 t^2 \left(x^2+1\right)-x\right)+t \left(x^2+1\right)\right)=0\label{polyproof 1}
\end{split}
\end{align}
This is a polynomial equation $P(f(x),B_1,B_3,B_4,t,x)=0$. We abbreviate it as $P(x)$. $\partial_{x_0}P(x)=0$ reads,
\begin{align}
\begin{split}
&2 B_1 t x^2+2 B_4 t \left(x^2+1\right) x-3 t^2x^2 \left(x^2+1\right)  \left(4 t^2 \left(x^2+1\right)-x\right)f(x)^2\\
&-2 t x\left(2 x^2+1\right)\left(4 t^2 \left(x^2+1\right)-x\right)f(x)-t^2 \left(5 x^4+6 x^2+1\right)+x^3=0\label{polyproof 2}
\end{split}
\end{align}
\eqref{polyproof 2} contains only positive powers in $t$. Let $t\to 0$, \eqref{polyproof 2} equals $x^3$. By \cref{theorem 2}, $\partial_{x_0}P(x)=0$ has three roots $X_1,X_2,X_3$.

$\partial_{x}P(x)=0$ reads,
\begin{align}
\begin{split}
    &+B_1 x (4t f(x)+1)- B_3 x++2 B_4\left(f(x) \left(3 t x^2+t\right)+x\right)\\
    &t^2 x \left(x \left(5 x^2+3\right)-8 t^2 \left(3 x^4+4 x^2+1\right)\right) f(x)^3+2 t \left(-2 t^2 \left(10 x^4+9 x^2+1\right)+4 x^3+x\right)f(x)^2\\
    &+x  \left(3 x-4 t^2 \left(5 x^2+3\right)\right)f(x)-t \left(3 x^2+1\right)=0
\end{split}
\end{align}
We have $3\times 3=9$ equations,
\begin{align}
\begin{split}
    &P(f(X_i),B_1,B_3,B_4,t,X_i)=0\\
   & \partial_{x_0}P(f(X_i),B_1,B_3,B_4,t,X_i)=0\\
   & \partial_x P(f(X_i),B_1,B_3,B_4,t,X_i)=0
\end{split}
\end{align}
for $i=1,2,3$. The $9$ unknowns are $f(X_i),X_i$($i=1,2,3$) and $B_1,B_3,B_4$.

We want to apply \cref{theorem 4} to prove that all unknowns are algebraic over the field generated by $\mathbb{C}(t)$, so we need to prove that $Det|J|$ defined in \cref{theorem 4} is not zero. Notice that this also implies that all roots are distinct.

We consider ordering the rows and columns of $J$ as follows,
\begin{small}
\begin{align}
    \begin{pmatrix}
        \partial_{x_0} P(X_1)& \partial_{x} P(X_1)& 0 & 0 & \dots &\partial_{B_1}P(X_1) &\partial_{B_2}P(X_1) &\partial_{B_3}P(X_1)\\
        \partial_{x_0}\partial_{x_0} P(X_1)& \partial_{x}\partial_{x_0}P(X_1)& 0 & 0 & \dots &\partial_{B_1}\partial_{x_0}P(X_1) &\partial_{B_2}\partial_{x_0}P(X_1) &\partial_{B_3}\partial_{x_0}P(X_1)\\
        \partial_{x_0}\partial_{x} P(X_1)& \partial_{x}\partial_{x} P(X_1)& 0 & 0 & \dots &\partial_{B_1}\partial_{x}P(X_1) &\partial_{B_2}\partial_{x}P(X_1) &\partial_{B_3}\partial_{x}P(X_1)\\
        0& 0 &\partial_{x_0} P(X_2)& \partial_{x} P(X_2)& \dots &\partial_{B_1}P(X_2) &\partial_{B_2}P(X_2) &\partial_{B_3}P(X_2)\\
        0 &0 &\partial_{x_0}\partial_{x_0} P(X_2)& \partial_{x}\partial_{x_0}P(X_2) & \dots &\partial_{B_1}\partial_{x_0}P(X_2) &\partial_{B_2}\partial_{x_0}P(X_2) &\partial_{B_3}\partial_{x_0}P(X_2)\\
         0 &0 & \partial_{x_0}\partial_{x} P(X_2)& \partial_{x}\partial_{x} P(X_2) & \dots &\partial_{B_1}\partial_{x}P(X_2) &\partial_{B_2}\partial_{x}P(X_2) &\partial_{B_3}\partial_{x}P(X_2)\\
         0&0&0&0&\dots&\dots&\dots&\dots
    \end{pmatrix}
\end{align}
\end{small}
Every $3k+1,3k+2,3k+3$ row are derivatives of $P(X_k),\partial_{x_0}P(X_k),\partial_{x}P(X_k)$. The column is indexed by derivatives to 
\[f(X_1),X_1,f(X_2),X_2,f(X_3),X_3,B_1,B_3,B_4.\] 
Notice that $\partial_{x_0}P(X_i)=\partial_x P(X_i)=0$ by the definition of $X_i$. Then as stated in \textbf{Theorem 3} in \cite{bousquet2016elementary}, $Det|J|$ factors into three blocks of size $2$ and one block of size $3$,
\begin{align}
    Det|J|=\pm\prod^3_{j=1}(\partial^2_{x_0x_0}P(X_j)\partial^2_{xx}P(X_j)-\partial^2_{x_0,x}P(X_j)^2)Det|\partial_{B_j}P(X_i)|_{1\leq i,j\leq 3}.
\end{align}
$Det|\partial_{B_j}P(X_i)|_{1\leq i,j\leq 3}\neq 0$ since a linear combination of $P(x),\partial_{x_0}P(x),\partial_xP(x)$ shows,
\begin{align}
    \begin{split}
        &B_3tx(x-1)+G_1(f(x),x,t)=0\\
        &B_1tx^2(x-1)^2(x+1)^2+G_2(f(x),x,t)=0\\
        &B_4tx(x+1)^2(x-1)^2+G_3(f(x),x,t)=0,
    \end{split}
\end{align}

where $G_1,G_2,G_3$ are the terms irrelevant to $B_1,B_3,B_4$. $x=0,1,-1$ are not the solutions $X_i$. $Det|\partial_{x_j}P(X_i)|_{1\leq i,j\leq 3}$ can be diagonalized and all values on the diagonal are not zero.

$(\partial^2_{x_0x_0}P(X_j)\partial^2_{xx}P(X_j)-\partial^2_{x_0,x}P(X_j)^2)\neq 0$ can be checked by direct calculation. Notice that,
\begin{align}
    &f(x)=F_{0,-1}+F_{1,-1}x+F_{2,-1}x^2+F_{3,-1}x^3+O(t^4)\\
    &B_1=F_{0,-1}-F_{2,-1}+O(t^4)\\
    &B_2=F_{1,-1}-F_{3,-1}+O(t^5).
\end{align}

Substitute the simulation results of $B_1,B_2,F_{0,-1}$ to $O(t^4)$ into  $(\partial^2_{x_0x_0}P(X_j)\partial^2_{xx}P(X_j)-\partial^2_{x_0,x}P(X_j)^2)$, we have,
\begin{align}
    (\partial^2_{x_0x_0}P(X_j)\partial^2_{xx}P(X_j)-\partial^2_{x_0,x}P(X_j)^2)=-9X_j^4+24(X_j^3+X_j^5)t^2+O(t^4).
\end{align}
We need to know the order of $t$ in $X_j(t)$ to make sure that this is not zero. So, we substitute the simulation results of $B_1,B_2,B_3, Hp_{-1}(x)$ to $\partial_{x_0}P(x)=0$,
\begin{align}
    \partial_{x_0}P=x^3+(-1-2x^2-x^4)t^2+O(t^4)=0.
\end{align}
By Newton Puiseux algorithm (or Newton polygon), we find the roots $X_j\sim t^{2/3}+O(t^{2/3})$. Thus,
\begin{align}
    (\partial^2_{x_0x_0}P(X_j)\partial^2_{xx}P(X_j)-\left(\partial^2_{x_0,x}P(X_j)\right)^2)\sim t^{8/3}+O\left(t^3\right).
\end{align}
which is not $0$. Then $Det|J|\neq 0$. By \cref{theorem 4}, $X_1,X_2,X_3,Hp_{-1}(X_i),Hp_{-1}(i),Hp_{-1}(-i),F_{0,-1}$ are algebraic over $\mathbb{C}(t)$. 

Now assume that we get the algebraic expression of $Hp_{-1}(x)$ via the general strategy. Substitute $Hp_{-1}(x)$ and $Hp_{-1}(\bar{x})$ in to \eqref{qudratic method 1} we solve $Hn(x)$, which reads,
\begin{align}
\begin{split}
    &Hn(x)+\frac{x}{3t}PR(x)=\frac{x}{3}F_{-1,0}+\frac{1}{3t\left(x+\bar{x}\right)}\\
    &+\frac{1}{3t}\sqrt{-4 t^2 x^3-4 t^2 x+x^2} \left(-2\bar{x} Hp_{-1}\left(\bar{x}\right)+xHp_{-1}(x)-\frac{1}{t\left(x^2+1\right)}\right)\label{Hn}.
\end{split}
\end{align}
$F_{-1,0}$ is solved by \eqref{Fm1}. It is a root of a quadratic equation whose coefficients are algebraic functions in $t$ (namely, $Hp_{-1}(i)$ and $Hp_{-1}(-i)$).  $Hn(x)$ is a sum of algebraic functions (right hand-side of \eqref{Hn}) and a D-finite term $\frac{xPR(x)}{3t}$. Thus, it is D-finite in $x,t$.

By a linear combination of \eqref{qudratic method 1} and \eqref{pos and neg}, we have
\begin{align}
\begin{split}
   &Hp(x)=-\frac{F_{-1,0}}{3 x}+\frac{Hp_{-1}(x)}{2 t}+\frac{PR(x)}{3 t x}+\frac{1}{6 t^2 x \left(x^2+1\right)}\\
   &+\sqrt{-4 t^2 x^3-4 t^2 x+x^2} \left(-\frac{Hp_{-1}\left(\bar{x}\right)}{3 t x^3}+\frac{Hp_{-1}(x)}{6 t x}-\frac{1}{6 t^2 x^2 \left(x^2+1\right)}\right).
\end{split}
\end{align}
$Hp(x)$ is a sum of some algebraic functions and a D-finite term $\frac{PR(x)}{3xt}$. It is D-finite.
\begin{remark}
Recall \eqref{condition}. $Hn(\bar{x})$ is written as an algebraic function $A_{-,0}(x)$ and $xQ(\bar{x},0)$. In \eqref{Hn}, $PR(x)$ comes from the orbit-sum of $F(x,y)$. We can also find it from the orbit sum of $Q(x,y)$ (quarter-plane models). We may take the $[y^1]$ terms of the corresponding \eqref{full orbit sum Model 32} of quarter-plane model,
\begin{align}
    xQ(x,0)-\bar{x}Q(\bar{x},0)=-\frac{(x-1) (x+1) Y_0}{t x}.
\end{align}
Then,
\begin{align}
    txQ(x,0)= -PR(x).
\end{align}
Thus, the algebraic equation $A_{-,0}(x)$ defined in \eqref{condition} is exactly the right hand-side of \eqref{Hn}. Although we cannot construct $A(x,y)$ for non-monomial models, we have $A_{-,0}(x)$.
\end{remark}

\subsection{More Algebraic Structures}\label{statement}
Besides $Hp_{-1}(x)$, this model has more algebraic properties. Let us consider substituting \eqref{pos and neg} into \eqref{RBVP for model 3}. But this time, we eliminate $Hp_{-1}(\bar{x})$ and $Hp(\bar{x})$. We will get the following equation,
\begin{align}
    \begin{split}
        & F_{-1,0}-2 x Hn\left(\bar{x}\right)-\frac{ Hn(x)}{x}-\frac{PR(x)}{t}+\frac{x^2}{t^2(x^2+1)}\\
        &+\sqrt{-4 t^2 x^3-4 t^2 x+x^2} \left(-\frac{ F_{-1,0}}{x}+\frac{ Hn(x)}{x^2}-2  Hp(x)+\frac{ PR(x)}{tx}-\frac{x}{t^2(x^2+1)}\right)
    \end{split}\label{Hn alg}
\end{align}
If we consider the substitution,
\begin{align}
    S(x)=\frac{Hn(x)}{x}+\frac{PR(x)}{3t}-\frac{1}{3}F_{-1,0}-\frac{1}{3t^3(x+\bar{x})},
\end{align}
\eqref{Hn alg} reads,
\begin{align}
\begin{split}
    &2  S\left(\bar{x}\right)+ S(x)=\\
    &\sqrt{-4 t^2 x^3-4 t^2 x+x^2} \left(-\frac{2  F_{-1,0}}{3 x}-\frac{2 \left(3 t^2 (x^2+1) Hp(x)+x\right)}{3t^2 \left(x^2+1\right)}+\frac{2  PR(x)}{3 t x}+\frac{ S(x)}{x}\right)\label{sx}
\end{split}
\end{align}
\eqref{sx} can be solved via exactly the same process as we solve $Hp_{-1}(x)$ in \eqref{qudratic method 1}. Although we still have a D-finite term $PR(x)$ in \eqref{sx}, if one carefully goes through the process, one may find that this $PR(x)$ does not appear in the final polynomial equation with one catalytic variable due to the trick of \eqref{qudratic method reflect}. So $S(x)$ is another algebraic function and should coincide with the solution \eqref{Hn}.

If we again remove $PR(x)$ in \eqref{sx} by \eqref{pos and neg}, we have
\begin{align}
    S(x)=\frac{2  Hn(x)}{3x}+\frac{2x}{3} Hp(x)-\frac{x}{3t} Hp_{-1}(x)-\frac{1}{3t ^2\left(x^2+1\right)}.
\end{align}
Thus, we find three independent equations.
\begin{enumerate}
    \item $Hp_{-1}(x)$ satisfies a polynomial equation with one catalytic variable.
    \item $\frac{2  Hn(x)}{3x}+\frac{2x}{3} Hp(x)-\frac{x}{3t} Hp_{-1}(x)-\frac{1}{3t ^2\left(x^2+1\right)}$ satisfies a polynomial equation with one catalytic variable.
    \item \eqref{pos and neg} is a linear equation between $Hn(x),Hp(x),Hp_{-1}(x)$ by a D-finite $PR(x)$.
\end{enumerate}
By \cref{orbit sum and null space}, $PR(x)$ can be obtained by $[y^0]\frac{OS(xy)}{K(x,y)}$. Thus if it equals $0$, all three independent equations are polynomial equations with rational coefficients in $x,t$. The solution is algebraic. This exactly proves the algebraic property of the model and the conjecture by Raschel and Trotignon in \cite{raschel2018walks}. They conjectured that for any finite group model, the three quadrant walks starting from $(-1,b)$ (or $(b,-1)$) have algebraic generating functions. This can be checked by the orbit-sum directly.

\section{The Combinatorial Riemann Boundary Value Problem in Matrix Form}\label{RBVP}
In this section, we theoretically analyze matrix cRBVP (combinatorial Riemann boundary value problem) in a general scheme and understand why the polynomial equations appear in the lattice walk in three quadrants.

Let us first review the integrability of the scalar RBVP and cRBVP. Consider $t$ as some fixed value. $x,\sigma(x)$ are two automorphism points on some Jordan curve $C$ with the condition $\sigma^2=Id$. $G(x,t)$ is H\"oder continues on $C$. Then on the curve $C$, the Riemann boundary value problem with Carlemann-shift is defined as,
\begin{align}
    H(\sigma(x),t)=G(x,t)H(x,t)+C(x,t).\label{generic RHP}
\end{align}
We may have some extra condition
\begin{align}
    &G(x,t)G(\sigma(x),t)=1\label{co1}\\
    &G(\sigma(x),t)C(x,t)+C(\sigma(x),t)=0\label{co2}.
\end{align}
\eqref{co1} is the condition that excludes the simple automorphism relation and \eqref{co2} is the solvability condition. If \eqref{co1} is not satisfied, we can solve \eqref{generic RHP} directly,
\begin{align}
    H(x,t)=\frac{G(\sigma(x),t)C(x,t)+C(\sigma(x),t)}{1-G(x,t)G(\sigma(x),t)}.
\end{align}

The Carelmann type Riemann boundary value problem is well studied in \cite{fayolle1999random,Litvinchuk2000}. Here, we are facing a combinatorial version of this problem. $H(x,t)$ is defined as a formal series in $\mathbb{C}[x][[t]]$, $G(x,t)\in \mathbb{C}[x,1/x][[t]]$ and $\sigma(x)\in \mathbb{C}[1/x][[t]]$. In \cite{xu2023combinatorial}, we proved that for a quarter-plane lattice walk problem, we can always reduce it to solving a functional equation of two unknowns $H_1(z,t)$ and $H_2(1/z,t)$. So, without loss of generality, we consider $\sigma(x)=1/x$. We recover the cRBVP defined in \eqref{generic cRBVP},
\begin{align}
    H(1/x,t)=G(x,t)H(x,t)+C(x,t).\label{generic RHP com}
\end{align}

In one-dimensional case, there is a standard approach to the solution using Gessel's canonical factorization \cite{gessel1980factorization}. Intuitively, if $G(x,t)$ admits a factorization, $G(x)=x^k G_-G_0G_+$ where $G_-\in\mathbb{C}[1/x][[t]]$, $G_+\in \mathbb{C}[x][[t]]$ and $G_0\in \mathbb{C}[[t]]$.
Then \eqref{generic RHP com} reads,
\begin{align}
   (G_-)^{-1} H(1/x,t)=x^kG_0G_+H(x,t)+(G_-)^{-1}C(x,t)\label{RHP solve}.
\end{align}
Except some finite monomials, $H(1/x,t)/G_-$ is a formal series in $1/x$ and and $G_0G_+H(x,t)$ is a formal series in $x$. Taking $[x^>]$ and $[x^<]$ will give two separate linear equations of $H(x,t)$ and $H(1/x,t)$.

The factorization can be obtained by a $log-exp$ procedure,
\begin{align}
    G(x,t)=e^{\log G(x,t)}=e^{\log x^kt^i}e^{[x^>]\log G(x,t)}\times e^{[x^0]\log G(x,t)}\times e^{[x^<]\log G(x,t)}.
\end{align}
For a well-defined $G(x,t)$ in lattice walk problems, such factorization always exists if we consider the Laurent expansion in some analytic domain with fixed branch (See Section.5.1 in \cite{xu2023combinatorial}). $G_-,G_+,G_0$ are in general D-algebraic functions \cite{xu2022interacting}.

\subsection{Birkhoff Factorization}
\eqref{generic RHP com} is always solvable as a scalar cRBVP. Now we treat \eqref{generic RHP com} as a matrix equation. For matrices, the $log-exp$ procedure is not applicable since matrices do not commute. In \cite{birkhoff1909singular}, Brikhoff proposed a recursive way to factorize a matrix.

\begin{lemma}(Auxiliary lemma in \cite{birkhoff1909singular})\label{birk lemma}
Let $\theta_{ij}(x)$ be a set of $n^2$ functions of $x$. single-valued and analytic for $|x|>R$, but not necessarily analytic at $x=\infty$ and also $Det|\theta_{ij}(x)|\neq 0$. Then it is possible to choose $n^2$ mulyipliers $\lambda_{ij}$, analytic at $x=\infty$, that 
\begin{align}
    \sum_{i=1}^n\lambda_{ij}(x)\theta_{jk}(x)=x^{-m}\zeta_{ik}(x)\label{birkhoff}\\
    \sum_{j=1}^n\theta_{ij}(x)\lambda_{jk}(x)=x^{-m}\zeta_{ik}(x),
\end{align}
where $m$ is a fixed positive integer or zero and $\zeta$ are entire functions of x for which the determinant $|\zeta_{ik}(0)|\neq 0$.
\end{lemma}
In matrix form, \eqref{birkhoff} reads $\Lambda(1/x) \Theta=x^{-m}Z(x)$. The matrix $\Lambda=(\lambda_{ij})_{n\times n}$ satisfies,
\begin{align}
    \Lambda=[x^<](\theta^{-1}\Psi),
\end{align}
where $\Psi$ and $\Gamma$ are the solution of the following equation,
\begin{align}
\begin{split}
    &\Gamma=[x^\geq](\Theta^{-1}\Psi)\\
    &\Psi=Tx^{-m}+[x^{\leq-m}](\Theta \Gamma).
\end{split}
\end{align}
$T$ is an arbitrary constant matrix with $Det|T|\neq 0$. $\Psi$ and $\Gamma$ are construct by an infinite process,
\begin{align}
\begin{split}
    &\Psi_0=T x^{-m}\\
    &\Gamma_0=[x^>](\Theta^{-1}\Psi_0)\\
    &\Psi_1=Tx^{-m}+[x^{<-m}](\Theta^{-1}\Gamma_0)\\
    &\Gamma_1=[x^>](\Theta^{-1}\Psi_1)\\
    &\Psi_2=Tx^{-m}+[x^{<-m}](\Theta^{-1}\Gamma_1)\\
    &\dots \dots\label{iteration}.
\end{split}    
\end{align}
The condition of \cref{birk lemma} is suitable for our problems when $t$ is small enough. See \cite{birkhoff1909singular} for details\footnote{$\theta_{ij}$ is required to be analytic for $|x|>R$ such that the iteration \eqref{iteration} converges. This convergence is also satisfied if we interpret everything as a formal series in $t$.}. Thus, if $Det|G(x)|\neq 0$, $G(x)$ admits a Birkhoff factorization. By \cref{birk lemma} and the iterative process, we can factor $G(x,t)$ and obtain $n$ scalar cRBVP equations in the form of \eqref{RHP solve}. Then we solve the problem.

\subsection{Orbit Sum as an Integrable Condition for Brickhoff Factorization in Lattice Walks}\label{universal}
To recover the result of \cref{three quarter} and \cite{bousquet2016square,bousquet2023walks}, we need more considerations. Consider $G(x)$ as a $3\times 3$ matrix. It has a rational part $P_0(x,t)$ and a square root part $P_1(x,t)\sqrt{\Delta(x,t)}$. $H(x,t)^T$ is a column vector
 \[(H_1(x,t),H_2(x,t),H_3(x,t))^T.\]
\begin{align}
H(1/x,t)^T=\left(P_0(x,t)+P_1(x,t)\sqrt{\Delta(x,t)}\right)H(x,t)^T+C(x,t)^T,\label{independent}
\end{align}
\eqref{independent} reveals the form of the cRBVP appearing in three-quadrant lattice walks. Everything in this equation and everything defined later in this section are rational or 
 formal (or Puixes) series in $t$. We denote them in the ring $\mathbb{C}(x,t)^{fr}[[t]]$\footnote{Fractional power series is considered here since we do not reject $\sqrt{t}$ when we expand $\sqrt{\Delta}$ around $t=0$ with fixed branch. It does not affect the final result.}. $\sqrt{\Delta(x,t)})$ is invariant under $x\to 1/x$. All elements in $P_0(x,t),P_1(x,t)$ are rational in $x,t$. In the following calculations, we abbreviate $\sqrt{\Delta(x,t)}$ as $\sqrt{\Delta}$ and $H(x,t),P_0(x,t),P_1(x,t),C(x,t)$ and any functions of $x,t$ as $H(x),P_0(x),P_1(x),C(x)$ 
 .etc for connivance.

For \eqref{independent} to be a well defined matrix cRBVP, we further require,
\begin{align}
\begin{split}
    &\Big(P_0(x)+P_1(x)\sqrt{\Delta}\Big)
    \Big(P_0(1/x)+P_1(1/x)\sqrt{\Delta}\Big)=Id\\
    &\Big(P_0(1/x)+P_1(1/x)\sqrt{\Delta}\Big)
    \Big(P_0(x)+P_1(x)\sqrt{\Delta}\Big)=Id,\label{au}
\end{split}
\end{align}
and
\begin{align}
    \Big(P_0(1/x)+P_1(1/x)\sqrt{\Delta}\Big)C(x)^T+C(1/x)^T=0.
\end{align}
We first impose an extra condition that $P_0(x)$ is full rank but $P_1(x)$ is rank $2$ as the lattice walk example in \cref{three quarter}. Denote the base vector of the left null space of $P_1(x)$ as $v(1/x)$. One will see why it is parameterized by $1/x$ but not by $x$ later. Since $P_1(x)$ is rational in $x$, $v(1/x)$ can be rational in $x$. We have the following lemma,

\begin{lemma}\label{lemma 1}
    We can choose a suitable $v_L(x)$, such that
\begin{align}
    v_L(1/x)P_0(x)=v_L(x)\label{eq}.
\end{align}
\end{lemma}
\begin{proof}
\eqref{eq} is obtained from the condition \eqref{au}. The first equation of $\eqref{au}$ is interpreted as the following two equations,
\begin{align}
    &P_0(x)P_0(1/x)+P_1(x)P_1(1/x)\Delta=Id\label{au1},\\
    &P_0(x)P_1(1/x)+P_1(x)P_0(1/x)=0\label{au2}.
\end{align}
Applying $v(1/x)$ of $P_1(x)$ on \eqref{au2}, we have
\begin{align}
    v(1/x)P_0(x)P_1(1/x)=0.\label{au3}
\end{align}
This suggest $v(1/x)P_0(x)$ is inside the null space of $P_1(1/x)$. Since $P_1(x)$ is rank two, $v(1/x)P_0(x)$ shall be proportional to $v(x)$. We denote it as $k(x)v(x)$. Since $P_1(x) \in \mathbb{C}(x,t)$, we can choose $v(1/x)\in \mathbb{C}(x,t)$. $k(x)v(x)=v(1/x)P_0(x)$ is also in $\mathbb{C}(x,t)$.

Left multiply $v(1/x)$ on \eqref{au1}, we have,
\begin{align}
    k(x)k(1/x)v(1/x)=v(1/x).
\end{align}
Thus $k(x)k(1/x)=1$. Denote $k(x)=\frac{f(x)}{f(1/x)}$, then $v_L(1/x)=f(1/x)v(1/x)$. 

$k(x)=\frac{f(x)}{f(1/x)}$ can be treated as a scalar cRBVP $\log{f(x)}-\log{f(1/x)}=\log{k(x)}$ and there exists formal series solutions. In fact, due to the rationality of $k(x)$, $f(x)$ is rational in $x$.
\end{proof}
With the observation of \cref{lemma 1}, multiply $v_L(1/x)$ to the left of \eqref{independent}, we have,
\begin{align}
    v_L(1/x)H(1/x)^T=v_L(x)H(x)^T+v_L(1/x)C(x)^T\label{automorphism relation 11}.
\end{align}
\eqref{automorphism relation 11} is suitable for taking $[x^>]$ and $[x^<]$ part since everything in this equation is rational. We can clear the denominator by multiplication or apply \cref{lemma 3}. Then we have,
\begin{align}
    \begin{split}
        &v_1(1/x)H(1/x)^T=NR_1(1/x)\\
        &v_1(x)H(x)^T=PR_1(x).\label{const 1}
    \end{split}
\end{align}
$v_1(x)=v_L(x)$ and $PR_1(x), NR_1(1/x)$ are the functions obtained by taking $[x^>],[x^<]$ part of \eqref{automorphism relation 11}. We find a linear equation of $H_1(x),H_2(x),H_3(x)$.
\begin{remark}
    If the rank of $P_1$ is one, by \eqref{au3}, there are two linearly independent vectors in the null space of $P_1$ with coefficients in $\mathbb{C}(x,t)$. We have two automorphism relations,
\begin{align}
\begin{split}
    v_1(1/x)P_0(x)=k_1(x)w_1(x)\\
    v_2(1/x)P_0(x)=k_2(x)w_2(x).
\end{split}
\end{align}
$w_1(x),w_2(x)$ are two vectors in the null-space and they do not need to be proportional to $v_1(x),v_2(x)$. We have two separable equations in the form \eqref{automorphism relation 11} and they are independent. By taking $[x^>]$ and $[x^<]$ part, we find two linear equation of $H_1(x),H_2(x),H_3(x)$. We will see an example of this case in \cref{outside the quadrant}.
\end{remark}
\begin{remark}\label{p0}
    If we exchange the condition from $Det|P_0(x)|\neq 0, Det|P_1(x)|=0$ to $Det|P_0(x)|=0, Det|P_1(x)|\neq 0$, by similar discussion, we have
    \begin{align}
         v(1/x)H(1/x)^T=(k(x)\sqrt{\Delta}) v(x)H(x)^T+v(1/x)C(x)^T.
    \end{align}
In addition, by \eqref{au1}, $ k(x)k(1/x)\Delta= 1$. This equation is a scalar cRBVP and by canonical factorization, it provides a linear equation of $H_1(x),H_2(x),H_3(x)$.
\end{remark}

\eqref{eq} in \cref{lemma 1} is exactly the full orbit-sum \eqref{orbit sum and null space 111} in lattice walk problems. An observation is,
\begin{corollary}\label{cor OS}
    In a lattice walk problem, the generating function is determined by a matrix cRBVP \eqref{independent}. If the full orbit-sum is section free, $P_1(x)$ is not full rank. 
\end{corollary}

\subsection{Solutions for the Eigenspace of Different Eigenvectors}\label{different eigenvalues}
From previous discussions, we notice that to find an equation in the form $v(1/x)P_0(x)=k(x)v(x)$, we shall consider the eigenspace of $P_0(x)P_0(1/x)$. Now, assume that we have the first eigenvalue $0$ and eigenvector $v_1(1/x)$ for $P_1(x)P_1(1/x)$, let us consider the remaining subspace.
\begin{lemma}\label{factor lemma}
\qquad 

\begin{enumerate}
    \item The matrix $P_0(x)P_0(1/x)$ and $P_1(x)P_1(1/x)$ ($P_0(1/x)P_0(x)$ and $P_1(1/x)P_1(x)$) share the same eigenvectors.
        \item Denote the eigenvalues of $P_0(x)P_0(1/x)$ as $\lambda_i(1/x)$, its eigenvectors as $v_1(x),v_2(x)$, $v_3(x)$. Assume $\lambda_i(x)\neq \lambda_j(x)$, $i,j=2,3$.
        \begin{enumerate}
            \item When $\lambda_i(x)\neq \lambda_i(1/x)$, we have,
        \begin{align}  
        \begin{split}
        &v_i(x)P_1(1/x)=m_{j1}(1/x)v_j(1/x)\\
        &v_i(x)P_0(1/x)=m_{j0}(1/x)v_j(1/x).\label{factor lemma 1}
        \end{split}
        \end{align}
        \item When $\lambda_i(x)= \lambda_i(1/x)$, we have,
        \begin{align} 
        \begin{split}
        &v_i(x)P_1(1/x)=m_{i1}(1/x)v_i(1/x)\\
        &v_i(x)P_0(1/x)=m_{i0}(1/x)v_i(1/x)\label{factor lemma 2}.
        \end{split}
        \end{align}
        \end{enumerate}
\end{enumerate}
\end{lemma}
\begin{proof}
Denote the eigenvalues of $P_0(1/x)P_0(x)$ as $\lambda_i(x)$ and the eigenvalues of $P_1(1/x)P_1(x)$ as $\mu_i(x)$. By \eqref{au1}, every eigenvector of $P_0(x)P_0(1/x)$ is an eigenvector of $P_1(x)P_1(1/x)$ and,
\begin{align}
    \lambda_i(x)+\mu_i(x)\Delta=1.
\end{align}
This is the first statement.

Now assume all eigenvalues are distinct. Consider the left eigenvector of $P_0(x)P_0(1/x)$ as $v_i(1/x)$ and right eigenvectors as $u_i(1/x)^T$, the equation reads
    \begin{align}
    \begin{split}
        &v_i(1/x)P_0(x)P_0(1/x)=\lambda_i(1/x)v_i(1/x)\\
        &P_0(x)P_0(1/x)u_i(1/x)^T=\lambda_i(1/x)u_i(1/x)^T\\
        &v_i(1/x)P_1(x)P_1(1/x)=\mu_i(1/x)v_i(1/x)\\
        &P_1(x)P_1(1/x)u_i(1/x)^T=\mu_i(1/x)u_i(1/x)^T\label{eighenvectors}.
    \end{split}
    \end{align}
\eqref{eighenvectors} also holds with $x\to 1/x$.

 Multiple $P_0(1/x)$ to the left of \eqref{au2} we have,
\begin{align}
P_0(1/x)P_0(x)P_1(1/x)+P_0(1/x)P_1(x)P_0(1/x)=0.\label{au4}
\end{align}
Notice that $P_0(1/x)P_1(x)=-P_1(1/x)P_0(x)$ since \eqref{au2} also holds with $x\to 1/x$. Substitute this into \eqref{au4}, we have,
\begin{align}
    P_0(1/x)P_0(x)P_1(1/x)-P_1(1/x)P_0(x)P_0(1/x)=0\label{au5}.
\end{align}
Multiply $v_i(x)$ on the left and $u_j(1/x)^T$ on the right to \eqref{au5}, we have,
\begin{align}
     (\lambda_i(x)-\lambda_j(1/x))v_i(x)P_1(1/x)u_j(1/x)^T=0\label{au6}.
\end{align}
We can multiply $P_1(x)$ to the right of \eqref{au2} and by similar calculation, we get
\begin{align}
     P_1(1/x)P_1(x)P_0(1/x)-P_0(1/x)P_1(x)P_1(1/x)=0\label{au52}.
\end{align}
This shows,
\begin{align}
     (\mu_i(x)-\mu_j(1/x))v_i(x)P_0(1/x)u_j(1/x)^T=0\label{au7}.
\end{align}
Since $\mu_i(x)=\frac{1-\lambda_i(x)}{\Delta}$, the factor $(\mu_i(x)-\mu_j(1/x))$ is the same as $(\lambda_i(x)-\lambda_j(1/x))$. These two equations both hold after $x\to 1/x$ and for any $i,j$.

Consider $\lambda_i(x)\neq \lambda_i(1/x)$. Then $v_i(x)P_1(1/x)u_i(1/x)^T=0$. Since the vectors orthogonal to $u_i(1/x)^T$ is spanned by $v_j(1/x),i\neq j$, we have the following,
\begin{align}  
&v_i(x)P_1(1/x)=m_{j1}(1/x)v_j(1/x)+n_{j1}(1/x)v_1(1/x)\label{au 80}\\
 &v_i(x)P_0(1/x)=m_{j0}(1/x)v_j(1/x)+n_{j0}(1/x)v_1(1/x)\label{au 90}.
\end{align}
Multiply $P_1(x)$ to the right of \eqref{au 80}, since $v_1(1/x)P_1(x)=0$, we have
\begin{align}
    v_i(x)P_1(1/x)P_1(x)=\mu_i(x)v_i(x)=m_{j1}(1/x)v_j(1/x)P_1(x)\label{au 801}.
\end{align}
Again multiply $P_1(1/x)$ to the right of the second equality of \eqref{au 801}, we have,
\begin{align}
    \mu_i(x)v_i(x)P_1(1/x)=\mu_i(x)m_{j1}(1/x)v_j(1/x)+\mu_i(x)n_{j1}(1/x)v_1(1/x)=\mu_j(1/x)m_{j1}(1/x)v_j(1/x).\label{au 801ex}
\end{align}

The right most term is the result of the right most term of \eqref{au 801}. The second equality of \eqref{au 801ex} means $n_{j1}(1/x)=0$ and $\mu_i(x)=\mu_j(1/x)$. By symmetry, $\mu_i(1/x)=\mu_j(x)$ and
\begin{align}
\begin{split}
&v_i(x)P_1(1/x)=m_{j1}(1/x)v_j(1/x)\\
&v_j(x)P_1(1/x)=m_{i1}(1/x)v_i(1/x)\label{au 802}.
\end{split}
\end{align}
So in this case, $\mu_i(x)=\mu_j(1/x)=m_{j1}(1/x)m_{i1}(x)$. Multiply $ P_0(x)$ to the right hand-side of \eqref{au 802} and substitute in the relation $P_1(1/x)P_0(x)=-P_0(1/x)P_1(x)$ and \eqref{au 90}, we find $n_{j0}(1/x)=0$ as well. It further implies $\lambda_i(x)=m_{j0}(1/x)m_{i0}(x)$ and $m_{j0}(1/x)m_{i1}(x)+m_{j1}(1/x)m_{i0}(x)=0$.

Then, consider the case $ \lambda_i(x)= \lambda_i(1/x)$. This implies $\lambda_i(x)\neq \lambda_j(1/x)$, otherwise $\lambda_i(x)=\lambda_j(x)$, which is excluded in the condition of the lemma. By \eqref{au6}, $v_i(x)P_1(1/x)u_j(1/x)^T=0$. Since the vectors orthogonal to $u_j(1/x)^T$ is spanned by $v_i(1/x),i\neq j$, we have the following result,
\begin{align}  
&v_i(x)P_1(1/x)=m_{i1}(1/x)v_i(1/x)+n_{i1}(1/x)v_1(1/x)\label{au 8}\\
 &v_i(x)P_0(1/x)=m_{i0}(1/x)v_i(1/x)+n_{i0}(1/x)v_1(1/x)\label{au 9}.
\end{align}
We claim $n_{i1}(1/x)=0$. Multiply $P_1(x)$ to the right of \eqref{au 8} and by the symmetric equation of \eqref{au 8} with $x\to 1/x$, we have,
\begin{align}
\begin{split}
    &\mu_i(x)v_i(x)=v_i(x)P_1(1/x)P_1(x)
    =m_{i1}(1/x)v_i(1/x)P_1(x)+n_{i1}(1/x)v_1(1/x)P_1(x)\\
    &=m_{i1}(1/x)\Big(m_{i1}(x)v_i(x)+n_{i1}(x)v_1(x)\Big)+n_{i1}(1/x)v_1(1/x)P_1(x).\label{au 10}
\end{split}
\end{align}
The last term $v_1(1/x)P_1(x)=0$ since $v_1(1/x)$ is the null-space of $P_1(x)$. Thus, we have $n_{i1}(x)=0$ and $\mu_i(x)=m_{i1}(x)m_{i1}(1/x)=\mu_i(1/x)$ which automatically excludes the previous case.
Multiplying $P_0(x)$ to \eqref{au 9}, by the same approach, we have $\lambda_i(x)=\lambda_i(1/x)=m_{i0}(x)m_{i0}(1/x)$.

Further, we consider $v_i(x)\to v_i(x)+a(x)v_1(x)$. We have,
\begin{align}
    (v_i(x)+a(x)v_1(x))P_0(1/x)=m_{i0}(1/x)v_i(1/x)+n_{i0}(1/x)v_1(1/x)+a(x)v_1(1/x)\label{au extra}.
\end{align}
We choose $a(x)$ satisfying the following scalar automorphism relation,
\begin{align}
    n_{i0}(1/x)+a(x)=m_{i0}(1/x)a(1/x).\label{au extra2}
\end{align}
Notice that $m_{i0}(1/x)m_{i0}(x)=\lambda_i(x)\neq 1$, this is a simple automorphism relation (not scalar cRBVP) and we have a rational solution. Then, $v_i(x)+a(x)v_1(x)$ is the new $v_i(x)$ such that $n_{i0}(1/x)=0$.
\end{proof}
We can multiply $v_i(1/x)$ to \eqref{independent} on the left, by \eqref{factor lemma 1} and \eqref{factor lemma 2}, we have,
\begin{align}
\begin{split}
    v_i(1/x)H(1/x)^T=(m_{j0}(x)+m_{j1}(x)\sqrt{\Delta})v_j(x)H(x)^T+v_i(1/x)C(x)^T\label{au 10a},
\end{split}
\end{align}
or
\begin{align}
\begin{split}
    v_i(1/x)H(1/x)^T=(m_{i0}(x)+m_{i1}(x)\sqrt{\Delta})v_i(x)H(x)^T+v_i(1/x)C(x)^T\label{au 10b}.
\end{split}
\end{align}
If $v_i(x),v_j(x)$ are rational in $x$, $v_i(1/x)H(1/x)\in \mathbb{C}(1/x)^{fr}[[t]]$ and $v_i(x)H(x)\in \mathbb{C}(x)^{fr}[[t]]$ ($v_j(x)H(x)\in \mathbb{C}(x)^{fr}[[t]]$). \eqref{au 10b} is a scalar cRBVP and we can solve $v_i(x)H(x)$. This gives the second and third linear equations. Besides, eigenvectors can be calculated from the eigenvalues directly. The property of $v_i(x)$ is determined by the eigenvalues $\lambda_i(x)$ but not $m_{i0}(x)$ or $n_{i0}(x)$. Integrability depends on $\lambda_i(x)$. 

Actually, we do not need to choose $v_i(x)$ such that $n_{i0}=0$ or $n_{i1}=0$. If they are not zero, \eqref{au 10b} reads,
\begin{align}
\begin{split}
    v_i(1/x)H(1/x)^T=(m_{i0}(x)+m_{i1}(x)\sqrt{\Delta})v_i(x)H(x)^T+v_i(1/x)C(x)^T+n_{i0}v_1(x)H(x)^T\label{au 102}.
\end{split}
\end{align}
$v_1(x)H(x)^T=PR(x)$ is solved by the linear equation of the null space. It is a known term in this equation.

\begin{remark}
    The condition $\mu_1(x)=0$ is not essential. Recall \eqref{au6}. Take the case  $\lambda_1(x)=\lambda_1(1/x)$ as an example. $v_1(x)P_1(1/x)u_j(1/x)^T=0$ for $j= 2,3$. This means that $v_1(x)P_1(1/x)$ is simultaneously in the vector space spanned by $\{v_1(1/x),v_2(1/x)\}$ and $\{v_1(1/x),v_3(1/x)\}$. Thus, $v_1(x)P_1(1/x)=m_{11}(1/x)v_1(1/x)$. For systems with distinct eigenvalues, the pivotal condition for integrability is the field extension of $\lambda$. \eqref{au 102} is separable if and only if $\lambda\in \mathbb{C}(x,t)$.
\end{remark}

\begin{remark}
There is another simple idea to find a separable equation. Rewrite the matrix cRBVP as,
    \begin{align}
        H(1/x)^T=\left(P_0(x)-k(x)P_1(x)\right)H(x)^T+\left(\sqrt{\Delta}+k(x)\right)P_1(x)H(x)^T+C(x)^T.\label{6.9}
    \end{align}
If $P_1(x)\neq 0$, we can always find $k(x)$, such that $Det|\left(P_0(x)-k(x)P_1(x)\right)|=0$. If $k(x)$ is rational, then by \cref{lemma 1}, we find a separable equation.

Denote $v(1/x)$ as the left null space vector of $\left(P_0(x)-k(x)P_1(x)\right)$. We have,
\begin{align}
\begin{split}
    &v(1/x)P_0(x)=v(1/x)k(x)P_1(x).\label{more remark 1}
\end{split}
\end{align}

Apply $v(1/x)$ to the left of \eqref{au1}, \eqref{au2} and substitute \eqref{more remark 1} in, we have,
\begin{align}
\begin{split}
    &v(1/x)k(x)P_1(x)P_0(1/x)+\Delta v(1/x)P_1(x)P_1(1/x)=v(1/x)\\
    &v(1/x)k(x)P_1(x)P_1(1/x)+v(1/x)P_1(x)P_0(1/x)=0.
\end{split}
\end{align}
A linear combination shows,
\begin{align}
    (-k(x)^2+\Delta )v(1/x)P_1(x)P_1(1/x)=v(1/x).
\end{align}
Thus, $v(1/x)$ is an eigenvector of $P_1(x)P_1(1/x)$ with eigenvalue $\frac{1}{(-k(x)^2+\Delta)}$. Notice that $k(x)$ can be a square root while $\frac{1}{(-k(x)^2+\Delta)}$ is still rational. The result we find from the eigenspace of $P_0(x),P_1(x)$ is contained in the eigenspace of $P_0(x)P_0(1/x),P_1(x)P_1(1/x)$.
\end{remark}

\subsection{Jordan Case}\label{Jordan Case}
Now consider $\lambda_i(x)=\lambda_j(x)=\lambda(x)$. Let us first consider the case that the subspace of the eigenvalue $\lambda(x)$ is a Jordan block. There are row vectors $u_L(x),v_L(x)$ and column vectors $u_R(x)^T,v_R(x)^T$, such that,
\begin{align}
\begin{split}
    &u_L(x)P_0(1/x)P_0(x)=\lambda(x)u_L(x)+v_L(x)\\
    &v_L(x)P_0(1/x)P_0(x)=\lambda(x)v_L(x)\\
    &P_0(1/x)P_0(x)u_R(x)^T=\lambda(x)u_R(x)^T+v_R(x)^T\\
    &P_0(1/x)P_0(x)v_R(x)^T=\lambda(x)v_R(x)^T. \label{Jordan relation}
\end{split}
\end{align}
Multiply $u_L(x),v_L(x)$ and $u_R(x)^T,v_R(x)^T$ one by one to left and right of \eqref{au1}. Apply \eqref{Jordan relation}. Direct calculation shows that $v_L(x),v_R(x)^T$ are still the eigenvectors of $P_1(1/x)P_1(x)$ and $-u_L(x)^T$, $-u_R(x)^T$ are the corresponding vector of $P_1(1/x)P_1(x)\Delta(x)$. Namely,
\begin{align}
    -u_L(x)P_1(1/x)P_1(x)\Delta(x)=(1-\lambda)(-u_L(x))+v_L(x).
\end{align}
Multiply $u_L(x),v_R(x)^T$ simultaneously to the left and right of $P_0(1/x)P_0(x)$,
\begin{align}   
\lambda(x)u_L(x)v_R(x)^T+v_L(x)v_R(x)^T=u_L(x)P_0(1/x)P_0(x)v_R(x)^T=\lambda(x)u_L(x)v_R(x)^T\label{left right}.
\end{align}
The first equality is the result by acting $P_0(1/x)P_0(x)$ on the left eigenvector and the last equality is the result by acting $P_0(1/x)P_0(x)$ on the right eigenvector. \eqref{left right} shows $v_L(x)v_R(x)^T=0$. Multiply $v_L(x)$ to the left and $v_R(1/x)^T$ to the right of \eqref{au5} \eqref{au52}, and by similar calculation in \cref{factor lemma}, we have
\begin{align}
\begin{split}
(\lambda(x)-\lambda(1/x))v_L(x)P_1(1/x)v_R(1/x)^T=0\\
(\mu(x)-\mu(1/x))v_L(x)P_0(1/x)v_R(1/x)^T=0.\label{field extension restrict}
\end{split}
\end{align}
This again is an equation in the form of \eqref{au6}. We first assume $\lambda(x)\neq \lambda(1/x)$ and by the same discussion as \eqref{au 8}, \eqref{au 9}, we have $\lambda(x)=m_0(x)m_0(1/x)$. Thus, $\lambda(x)=\lambda(1/x)$, a contradiction. Again, multiply $u_L(x)$ and $v_R(1/x)^T$ to the left and right of \eqref{au5} and \eqref{au52}. Since $\lambda(x)=\lambda(1/x)$, we have
\begin{align}
\begin{split}
&\Big(\lambda(x)u_L(x)+v_L(x)\Big)P_1(1/x)v_R(1/x)^T-\lambda(x)u_L(x)P_1(1/x)v_R(1/x)^T=0\\
&\frac{1}{\Delta(x)}\Big((1-\lambda)u_L(x)-v_L(x)\Big)P_0(1/x)v_R(1/x)^T-\frac{1-\lambda}{\Delta(x)}u_L(x)P_0(1/x)v_R(1/x)^T=0.
\end{split}
\end{align}
So $v_L(x)P_1(1/x)v_R(1/x)^T=v_L(x)P_0(1/x)v_R(1/x)^T=0$. This gives exactly the same result as \eqref{au 8} and \eqref{au 9}. we still have $\lambda(x)=m_0(x)m_0(1/x)$ and,
\begin{align}
    \begin{split}
 &v_L(x)P_1(1/x)=m_1(x)v_L(1/x)+n_1(x)v_1(1/x)\\
 &v_L(x)P_0(1/x)=m_0(x)v_L(1/x)+n_1(x)v_1(1/x).
    \end{split}
\end{align}

We again get \eqref{au 10b} and solve $v_L(x)H(x)^T=PR_L(x)$. Although we have not found the third equation, two equations $v_L(x)H(x)^T=PR_L(x),v_1(x)H(x)^T=PR_1(x)$ are enough to solve the matrix cRBVP.
\begin{remark}
    The Jordan form case is the most complicated case in linear algebra, but it is the simplest case in our problem. Since $\lambda(x)$ is a multiple root of a characteristic polynomial in $\mathbb{C}(x)^{fr}[[t]]$, $\lambda(x)$ is also in $\mathbb{C}(x)^{fr}[[t]]$. Thus, each term in $v_L(x)$ is rational in $x$. We can always solve this subspace.
\end{remark}

\subsection{Algebraic Solutions in $\lambda_i=\lambda_j=1/4$ Case}\label{equal1}
The cases we discussed above are still following the idea of factorization. $v_2(1/x),v_3(1/x)$ are the same for $P_1(x),P_0(x)$. We reduce the $3\times 3$ matrix cRBVP into three independent scalar cRBVP and this is equivalent to matrix factorization. Now consider the case $\lambda_2(x)=\lambda_3(x)=\lambda(x)$ and the matrix is still diagonalizable. $v_2(1/x),v_3(1/x)$ span a two-dimensional subspace. If $\lambda(x)\neq \lambda(1/x)$, we can choose orthogonal base vectors $v_i(x)u_j(x)^T=0$. Similar argument as \eqref{au 10} in \cref{factor lemma} shows $\lambda(x)=m_0(x)m_0(1/x)$, a contradiction. Thus, $\lambda(x)=\lambda(1/x)=m_0(x)m_0(1/x)$.

In this case, we claim that there exist vectors $v(x)$ in this subspace such that $v(x)P_0(1/x)=m_0(x)v(1/x)$ but $v(x)P_1(1/x)\neq m_1(x)v(1/x)$. To construct such $v(x)$, first choose an arbitrary vector $v_3(x)$ in the subspace. In addition, we choose $v_3(x)$ rational in $x$. If $v_3(x)P_0(1/x)=m_0(1/x)v_3(x)$, it is the required $v(x)$. If not, we choose $v_2(x)$ by,
\begin{align}
    v_2(x)P_0(1/x)=k(1/x)v_3(1/x).\label{au 12}
\end{align}
We can always find $v_2(x)$ since $P_0(1/x)$ is full rank. Multiply $P_0(x)$ to the right of \eqref{au 12}, we have
\begin{align}
\lambda(x) v_2(x)=v_2(x)P_0(1/x)P_0(x)=k(1/x)v_3(1/x)P_0(x).
\end{align}
This shows,
\begin{align}
    v_3(x)P_0(1/x)=\frac{\lambda(x)}{k(x)}v_2(1/x)\label{au 13}.
\end{align}
We choose $k(x)$ such that $\lambda(x)=k(x)k(1/x)$. Then $\eqref{au 12}+\eqref{au 13}$ gives $v(x)=v_2(x)+v_3(x)$ and $k(x)=m_0(x)$. Thus, for an arbitrary $v_3(x)$ in the eigenspace of $P_0(1/x)P_0(x)$, we can find $v(x)$, 
\begin{align}
    v(x)=v_3(x)+\frac{1}{m_0(x)}v_3(1/x)P_0(x)\label{construction of v(x)}.
\end{align}
such that $v(x)P_0(1/x)=m_0(x)v(1/x)$ but $v(x)P_1(1/x)$ may not be equal to $m_1(x)v(1/x)$.

We further claim that everything is still rational in $x$. $v(x)$ is construct by a vector $v_3(x)$ rational in $x$, a matrix $P_0(1/x)$ rational in $x$ and $m_0(x)$. $m_0(x)$ is not determined to be rational. However, $\lambda(x)$ is rational and symmetric and $\lambda(x)=m_0(x)m_0(1/x)$. It factors by roots (for example, $x+1/x$ factored as $i(1-ix)(1-i/x)$). Thus, $m_0(x)$ is still rational in $x$.

Now, we may solve $v(x)H(x)^T$. We currently can only solve the cases with special values of $\lambda$, namely $\lambda=1/4$ and its `extensions'. $\lambda=1/4$ is the case of the three-quadrant walks in \cref{three quarter} (also the walks in \cite{bousquet2016square,bousquet2021more,bousquet2023walks}) and we will clarify the meaning of extensions in \cref{more algebraic solutions}.

The way of solving $v(x)H(x)^T$ is exactly the same as solving $Hp_{-1}(x)$ in \cref{three quarter}. We go through it in a general framework. Multiple $v(1/x)$ to the matrix equation \eqref{independent},
\begin{align}
    v(1/x)H(x)^T=(v(1/x)P_0(x)+v(1/x)P_1(x)\sqrt\Delta(x))H(x)^T+v(1/x)C(x)^T\label{independent 2}.
\end{align}
Since $v(1/x)P_0(x)=m_0(1/x)v(x)$ and by suitable arrangements, the equation reads
\begin{align}
    v(1/x)H(1/x)^T-m_0(x)v(x)H(x)^T+C_1(x)=\sqrt{\Delta}\left(v(1/x)P_1(x)H(x)^T+C_2(x)\right)\label{quadratic 1},
\end{align}
where $C_1(x),C_2(x)$ are constant function in $x$.

 Without loss of generality, let us assume $f(x)$ is the solution of
\begin{align}
    f(1/x)-m_0(x)f(x)=C_1(x)\label{auto 1}.
\end{align}
This is a simple automorphism relation since $m_0(x)m_0(1/x)\neq 1$. We can solve it directly and $f(x)\in \mathbb{C}(x)^{fr}[[t]]$. Denote $A(x)=v(x)H(x)^T+f(x)$ and $B(x)=v(x)p_1(x)H(x)^T+C_2(x)$, \eqref{quadratic 1} has a extremely simply form,
\begin{align}
    A(1/x)-m_0(x)A(x)=\sqrt\Delta B(x)\label{quadratic 2}.
\end{align}
Square \eqref{quadratic 2} we have,
\begin{align}
    A(1/x)^2-2m_0(x)A(x)A(1/x)+m_0(x)^2A(x)^2=\Delta B(x)^2\label{quadratic 22}.
\end{align}
If $m_0(x)=1/2$ (or $-1/2$) as the three-quadrant walk cases, we take the $[x^<]$ part of this equation,
\begin{align}
    A(1/x)^2-[x^<]A(x)A(1/x)=C_3(1/x),\label{general n}
\end{align}
and apply $x\to 1/x$ to \eqref{general n},
\begin{align}
    A(x)^2-[x^>]A(x)A(1/x)=C_3(x).\label{general p}
\end{align}
Due to the $x\to 1/x$ symmetry of $A(x)A(1/x)$ and the trivial relation,
\begin{align}
[x^>]A(x)A(1/x)+[x^0]A(x)A(1/x)+[x^<]A(x)A(1/x)=A(x)A(1/x),    
\end{align}
we get,
\begin{align}
    A(x)^2-A(x)A(1/x)+A(1/x)^2=C_4(x)\label{quadratic 3}.
\end{align}
The operator $[x^<]$ is applicable since coefficient in \eqref{quadratic 22} is rational. We can either clear the denominator by multiplication or applying \cref{lemma 3}. $C_3(x)$ is a known function in $x$ and is obtained by taking the $[x^<]$ part of \eqref{quadratic 22}. $B(x)$ is eliminated since it does not involve formal series in $1/x$. $C_4(x)$ consists of $C_3(x)+C_3(1/x)$ and $[x^0]$ of $A(x)A(1/x),A(x^2),A(1/x)^2,\Delta B(x)^2$. It is also a known function of $x$ with some unknown functions in $t$.

Multiply $A(x)+A(1/x)$ to \eqref{quadratic 3}, we have,
\begin{align}
    A(x)^3+A(1/x)^3=C_4(x)(A(x)+A(1/x)).\label{cubicm1}
\end{align}
Again by taking $[x^>]$ degree term of this equation, we find a polynomial equation with one Catalytic variable,
\begin{align}
    A(x)^3=C_4(x)A(x)+C_5(x).\label{cubic}
\end{align}
$C_5(x)$ is the extra terms known in $x$ by taking $[x^>]$ of \eqref{cubicm1}. $A(x)$ is an unknown function in $x,t$. $C_4(x),C_5(x)$ contains unknown functions in $t$. \eqref{cubic} is a polynomial equation with one catalytic variable defined in \eqref{polynomial equation} and we shall use the general strategy in \cite{bousquet2006polynomial} to solve $A(x)$.

Recall that $v_3(x)$ is chosen arbitrarily in the construction. We can choose another $u_3(1/x)$ such that such that $u(1/x)P_0(x)=m'_0(x)u(x)$. The choice of $m'_0(x)$ is $\pm m_0(x),\pm m_0(1/x)$. Let $u(x)H(x)^T=F(x)$. Then $F(x)$ satisfies an equation in the same form as \eqref{cubic}.

Combining \eqref{const 1}, and two \eqref{cubic}, we got three independent equations,
\begin{align}
    \begin{split}
        &v_1(x)H(x)^T=PR_1(x)\\
        &P_1(v(x)H(x)^T,A_1,\dots A_i,x,t)=0\\
        &P_2(u(x)H(x)^T,F_1,\dots F_j,x,t)=0\label{alg set}.
    \end{split}
\end{align}
$A_1,\dots,A_i$ and $F_1,\dots,F_j$ are the unknown functions in $t$ appearing in the calculation. We find three independent algebraic equations with coefficients in $x,t,PR_1(x)$, one is linear and two are polynomial. $v(x),u(x),v_1(x)$ span the entire solution space. Thus, the matrix cRBVP \eqref{independent} is integrable in this case.

For the equation set to be exactly solvable, we still need to prove that there are enough algebraic roots $X_i$ for $P_1$ and $P_2$. This is a case-by-case calculation. However, the property in $x$ is determined by \eqref{alg set}. The only D-finite term in $x$ is $PR_1(x)$. So if $PR_1(x)=0$, the solution is algebraic in $x$. Recall the results of Kreweras, reveres Kreweras and Gessel's walk in \cite{beaton2019quarter,bousquet2016elementary,bousquet2005walks} and \cref{cor OS}, this is equivalent to the zero orbit-sum condition in lattice walk problems.

\subsection{More Algebraic Solutions}\label{more algebraic solutions}
Let us conclude the algebraic case we have found so far,
\begin{itemize}
    \item cyclotomic case: $m_0(x)=\pm 1/2, \lambda(x)=1/4$. We use the trick $(a+b)(a^2-ab+b^2)=a^3+b^3$ to separate $A(x),A(1/x)$. $A(x)$ satisfies an algebraic equation of degree $3$.
\end{itemize}
The trick $(a+b)(a^2-ab+b^2)=a^3+b^3$ is not a unique property of the cyclotomic polynomial.

Recall \eqref{quadratic 22}. Multiply $m_0(1/x)$ to the equation, it reads,
\begin{align}
    m_0(1/x)A(1/x)^2-2m_0(x)m_0(1/x)A(x)A(1/x)+m_0(1/x)m_0(x)^2A(x)^2=m_0(1/x)\Delta B(x)^2\label{reflect}
\end{align}
If we take $[x^<]$ terms of \eqref{reflect} and apply the $x\to 1/x$ symmetry of $A(x)A(1/x)$ as we did in $\lambda=1/4$ case, we have,
\begin{align}
    m_0(1/x)A(1/x)^2-2m_0(x)m_0(1/x)A(x)A(1/x)+m_0(x)A(x)^2=C_4(x)\label{reflect 3},
\end{align}
We may solve $A(1/x)$ directly,
\begin{align}
    \begin{split}
        \frac{A(1/x)}{m_0(x)}=A(x)\pm\sqrt{\left(1-\frac{1}{m_0(1/x)m_0(x)}\right)A(x)^2+\frac{C_4(x)}{m_0(1/x)m_0(x)^2}}
    \end{split}.
\end{align}
Let us denote the solution as,
\begin{align}
    \frac{F(1/x)}{m_0(x)}=F(x)+ \epsilon\sqrt{kF(x)^2+C}\label{F}
\end{align}
$\epsilon=\pm 1$ and $(kF(x)^2+C)$ has argument in $(0,\pi)$. $k=1-1/\lambda(x)$. $C$ is rational in $x$. Now consider the module space spanned by $\{F(1/x),F(1/x)^1,\dots F(1/x)^n\}$ with coefficients in $\mathbb{C}(x)^{fr}[[t]]$,
\begin{align}
    \begin{split}
        &\frac{F(1/x)}{m_0(x)}=F(x)+ \epsilon\sqrt{kF(x)^2+C}\\
        &\left(\frac{F(1/x)}{m_0(x)}\right)^2=F(x)^2+ 2\epsilon F(x)\sqrt{kF(x)^2+C}+\epsilon^2(kF(x)^2+C)\\
        &\left(\frac{F(1/x)}{m_0(x)}\right)^3=F(x)^3+ 3\epsilon F(x)^2\sqrt{kF(x)^2+C}+ 3\epsilon^2 F(x)(kF(x)^2+C)+\epsilon^3(kF(x)^2+C)\sqrt{kF(x)^2+C}\\
        &\dots\\
        &\left(\frac{F(1/x)}{m_0(x)}\right)^n=F(x)^n+ \binom{n}{n-1}\epsilon F(x)^{n-1}\sqrt{kF(x)^2+C}\\
        &+ \binom{n}{n-2}\epsilon^2F(x)^{n-2}(kF(x)^2+C)+ \binom{n}{n-3}\epsilon^3F(x)^{n-3}(kF(x)^2+C)\sqrt{kF(x)^2+C}+\dots
    \end{split}.
\end{align}
The right hand-side of $\left(\frac{F(1/x)}{m_0(x)}\right)^n$ is spanned by $\{F(x)^i,F(x)^j\sqrt{kF(x)^2+C}\}$ and $i\leq n, j\leq n-1$. A general module generated by $\left(\frac{F(1/x)}{m_0(x)}\right)^n$ reads,
\begin{align}
P_1(F(1/x),x,t)=P_2(F(x),x,t)+P_3(F(x),x,t)\sqrt{kF(x)^2+C},
\end{align}
$P_1,P_2,P_3$ are three polynomials. In addition, comparing \eqref{F} and \eqref{quadratic 2}, we have \[\sqrt{\Delta}B(x)=\epsilon m_0(x)\sqrt{kF(x)^2+C}.\] There are two ways to find a separable equation.
\begin{enumerate}
    \item If the module spanned by $F(1/x)$ is spanned by $F(x)$ without $\sqrt{kF(x)^2+C}$, then $P_3(F(x),x,t)=0$. We can separate the $[x^>]$ and $[x^<]$ of the following equation,
\begin{align}
P_1(F(1/x),x,t)=P_2(F(x),x,t).
\end{align}
 Notice that by a linear combination, the lower degree terms, $F(x)^{j-1}\sqrt{kF(x)^2+C}$ with $j<n$ on the right hand-side of $\left(\frac{F(1/x)}{m_0(x)}\right)^n$, can be eliminated by the right hand-side of $\left(\frac{F(1/x)}{m_0(x)}\right)^j$, $j<n$. $F(x)^{n-1}\sqrt{kF(x)^2+C}$ only appears on the right hand-side of $\left(\frac{F(1/x)}{m_0(x)}\right)^n$.  Thus, $P_3(F(x),x,t)=0$ indicates that the coefficient of $F(x)^{n-1}\sqrt{kF(x)^2+C}$ equals $0$ for some $n$. Direct calculation shows, if $n=2N$.
\begin{align}
     \epsilon\binom{2N}{2N-1}+\epsilon^3\binom{2N}{2N-3}k+\epsilon^5\binom{2N}{2N-5}k^2\dots+\epsilon^{2N-1}\binom{2N}{1} k^{N-1}=0\label{1}.
\end{align}
If $n=2N-1$, we have,
\begin{align}
     \epsilon\binom{2N-1}{2N-2}+\epsilon^3 \binom{2N-1}{2N-4}k+\epsilon^5\binom{2N-1}{2N-6}k^2\dots+\epsilon^{2N-1}\binom{2N-1}{0} k^{N-1}=0\label{2}.
\end{align}

The solution is irrelevant to the choice of $\epsilon$ since $\epsilon^2=1$. If $n=3$, the equation reads $\binom{3}{1}+\binom{3}{3}k=0$. We immediately have $\lambda=1/4$, which is the case of three-quadrant walks.

    \item If the vector space spanned by $F(1/x)^n$ is spanned by $F(x)^j\sqrt{kF(x)^2+C}$, then,  
    
    $P_2(F(x),x,t)=0$. Due to the relation $\sqrt{\Delta}B(x)=\epsilon m_0(x)\sqrt{kF(x)^2+C}$, we have,
\begin{align}
P_1(F(1/x),x,t)=\sqrt{\Delta}P_3(F(x),x,t)\frac{B(x)}{\epsilon m_0(x)}.
\end{align}
$\Delta$ is a rational function and $\sqrt{\Delta}$ can be canonically factorized as,
\begin{align}
\sqrt{\Delta}=\sqrt{\Delta_0}\sqrt{\Delta_+}\sqrt{\Delta_-}=\sqrt{\Delta_0}x^{n/2}\prod_i(\sqrt{1-x/X_i})\prod_{j}(\sqrt{1-X_j/x}),
\end{align}
where $X_i$ are roots in $\mathbb{C}(x)^{fr}[[t]]$ and $X_j$ are roots such that $1/X_j\in \mathbb{C}(x)^{fr}[[t]]$. Then,
\begin{align}
\frac{\epsilon m_0(x)}{\sqrt{\Delta_-}}P_1(F(1/x),x,t)=\sqrt{\Delta_+\Delta_0}P_3(F(x),x,t)B(x).
\end{align}
is a suitable form for taking $[x^>]$ and $[x^<]$ series.

To achieve this, $P_2(F(x),x,t)=0$ indicates that the coefficient of $F(x)^n$ equals $0$ for some $n$. Similar argument as the previous case shows, if $n=2N$, we have,
\begin{align}
     \binom{2N}{2N}+\epsilon^2\binom{2N}{2N-2}k+\epsilon^4\binom{2N}{2N-4}k^2\dots+\epsilon^{2N}\binom{2N}{0} k^{N}=0\label{3}.
\end{align}
If $n=2N-1$, we have,
\begin{align}
     \binom{2N-1}{2N-1}+\epsilon^2 \binom{2N-1}{2N-3}k+\epsilon^4\binom{2N-1}{2N-5}k^2\dots+\epsilon^{2N-2}\binom{2N-1}{1} k^{N-1}=0\label{4}.
\end{align}

The solution is irrelevant to the choice of $\epsilon$. If $n=3$, the equation reads $\binom{3}{3}+\binom{3}{1}k=0$. We have $\lambda=-3/4$. 
\end{enumerate}
For $n=2N-1$, the solution $\hat{k}$ in case $2$ is $1/\hat{k}$ in case $1$. For $n=2N$, the solutions of these two cases are not related.

Thus, for all solutions $\hat{k}$ of \eqref{1}, \eqref{2}, \eqref{3}, \eqref{4}, we can find a polynomial equation in $F(x), F(1/x)$ and these two unknown functions are separated.  Taking $[x^<]$ will give a polynomial equation of $F(1/x)$ with algebraic coefficients in $x$. We find a polynomial equation in one catalytic variable for $F(x)$.

The choice of $F(x)$ is still arbitrary. Recall \eqref{construction of v(x)}. We can construct two independent vectors $v(x),u(x)$ and $v(x)H(x)^T,u(x)H(x)^T$ both satisfy a polynomial equation in one catalytic variable. So the cRBVP \eqref{independent} is integrable for $k=\hat{k}$.

\subsection{Algebraic Properties in General $\lambda_i=\lambda_j$ Case}\label{algebra equal case}
If we do not assume any special values of $\lambda$, we may see some algebraic properties. Again recall \eqref{reflect},
\begin{align}
    m_0(1/x)A(1/x)^2-2m_0(x)m_0(1/x)A(x)A(1/x)+m_0(1/x)m_0(x)^2A(x)^2=m_0(1/x)\Delta B(x)^2\label{reflect recall}.
\end{align}
Apply $x\to 1/x$,
\begin{align}
    m_0(x)A(x)^2-2m_0(x)m_0(1/x)A(x)A(1/x)+m_0(x)m_0(1/x)^2A(1/x)^2=m_0(x)\Delta B(1/x)^2.
\end{align}
Then, eliminate $2m_0(x)m_0(1/x)A(x)A(1/x)$ by linear combination,
\begin{align}
\begin{split}
    m_0(x)A(x)^2&-m_0(1/x)m_0(x)^2A(x)^2+m_0(1/x)\Delta B(x)^2\\
    =
    &m_0(x)\Delta B(1/x)^2-m_0(x)m_0(1/x)^2A(1/x)^2+ m_0(1/x)A(1/x)^2.\label{quadratic 5}
\end{split}
\end{align}
Despite some rational factors, this is still an equation such that the unknown functions of the formal power series in $x$ and the formal power series in $1/x$ are separated. Take the $[x^>]$ part of this equation, we have
\begin{align}
    m_0(x)A(x)^2-m_0(1/x)m_0(x)^2A(x)^2&+m_0(1/x)\Delta B(x)^2=C_6(x)\label{quadratic 4}.
\end{align}
$C_6(x)$ is the remaining known function in $x$ after taking the $[x^>]$ degree terms. This is an algebraic function of degree $2$ and gives a relation between $H_1(x),H_2(x),H_3(x)$. We cannot find another polynomial equation for them. $v(x)$ is chosen arbitrarily. However, if we choose another independent $u(x)$, we will get the same result as \eqref{quadratic 4}. This is because the equation of $x$ and the equation with $x\to 1/x$ span the whole subspace. \eqref{quadratic 4} is the algebraic property of this subspace.

\subsection{$\lambda_i\neq \lambda_j$ Case Revisit}\label{revisit}
Let us return to the $\lambda_i\neq \lambda_j$ case. If the eigenvalues are rational in $x$, \eqref{au 10a} is a separable linear equation. If the eigenvalues contain square roots, we cannot separate the $[x^>],[x^<]$ part. But still, there are some algebraic properties as \eqref{quadratic 4} in the $\lambda_i= \lambda_j$ case.

Suppose that there is a square root term $\sqrt{\delta}$ introduced in the eigenvalue $\lambda_i(x)$, then all $v_i(x)$, $m_{i0}(x),m_{i1}(x)$ contain $\sqrt{\delta}$ by direct calculations. In addition, since $\lambda_j(x)$ and $\lambda_i(x)$ are two conjugated roots of a quadratic polynomial, $v_i(x)$ and $v_j(x)$, $m_{i0}(x)$ and $m_{j0}(x)$ are all conjugate to each other.

Denote the Galois transform as $\sigma (\sqrt{\delta})\to -\sqrt{\delta}$. We first consider the case $\lambda_i(x)=\lambda_i(1/x)$. The following change of variables makes the equation simpler,
\begin{align}
    \begin{split}
        &v_i(1/x)H(1/x)^T=R(1/x)+I(1/x)\sqrt{\delta}\\
        &m_{i0}(x)=a(x)+b(x)\sqrt{\delta}\\
        &m_{i1}(x)=s(x)+r(x)\sqrt{\delta}\\
        &v_i(1/x)C(x,t)^T=J_1+J_2\sqrt{\delta}+J_3\sqrt{\Delta}+J_4\sqrt{\delta}\sqrt{\Delta}.
    \end{split}
\end{align}
$R(1/x),I(1/x)$ are the unknown functions. All known terms, $a(x),b(x),r(x),s(x)$, $\delta,\Delta$, $J_1,J_2,J_3,J_4$ are rational in $x$. Due to the existence of $\sqrt{\delta}$, the $[x^>]$ and $[x^<]$ parts of \eqref{au 10a} cannot be separated.

We construct a polynomial equation like \eqref{quadratic 5}. We first deal with the term $v_i(1/x)C(x,t)^T$. Consider the following equation,
\begin{align}
\begin{split}
    &F(1/x)+G(1/x)\sqrt{\delta}=\\
    &\Big((a(x)+b(x)\sqrt{\delta})+(s(x)+r(x)\sqrt{\delta})\sqrt\Delta\Big)\left(f(x)+g(x)\sqrt{\delta}\right)-\left(J_1+J_2\sqrt{\delta}+J_3\sqrt{\Delta}+J_4\sqrt{\delta}\sqrt{\Delta}\right)\label{change 1}.
\end{split}
\end{align}
If we want rational solutions of $F(1/x),G(1/x),f(x),g(x)$, we shall choose $f(x),g(x)$ such that coefficients of $\sqrt{\Delta}$ vanishes. This leads to the following equation,
\begin{align}
\left(s(x)+r(x)\sqrt{\delta}\right)\left(f(x)+g(x)\sqrt{\delta}\right)=(J_3+J_4\sqrt{\delta}).
\end{align}
This is equivalent to,
\begin{align}
    \left(
\begin{array}{c}
 J_3  \\
 J_4 \\
\end{array}
\right)
=
    \left(
\begin{array}{cc}
 s(x) &   r(x)\delta \\
 r(x) & s(x) \\
\end{array}
\right)
    \left(
\begin{array}{c}
 f(x)  \\
 g(x)  \\
\end{array}
\right).
\end{align}
This is a linear equation set and the determinant is $s(x)^2-\delta r(x)^2=m_{i1}(x)\overline{m_{i1}(x)}\neq 0$. Thus, the solutions $f(x),g(x)$ are rational in $x$. Substituting $f(x),g(x)$ into \eqref{change 1} and collecting the rational terms and the $\sqrt{\delta}$ terms, we solve $F(1/x),G(1/x)$.

Substitute the solution of \eqref{change 1} into \eqref{au 10a}, we have,
\begin{align}
\begin{split}
    &(R(1/x)+F(1/x))+(I(1/x)+G(1/x))\sqrt{\delta}\\
    &=\Big((a(x)+b(x)\sqrt{\delta})+(s(x)+r(x)\sqrt{\delta})\sqrt\Delta\Big)\left((R(x)+f(x))+(I(x)+g(x))\sqrt{\delta}\right)\label{change 2}.
\end{split}
\end{align}
This is the equation of $v_i(x)$. The equation of $v_j(x)$ is the conjugate of \eqref{change 2}
\begin{align}
\begin{split}
    &(R(1/x)+F(1/x))-(I(1/x)+G(1/x))\sqrt{\delta}\\
    &=\Big((a(x)-b(x)\sqrt{\delta})+(s(x)-r(x)\sqrt{\delta})\sqrt\Delta\Big)\left((R(x)+f(x))-(I(x)+g(x))\sqrt{\delta}\right)\label{change 3}.
\end{split}
\end{align}
The product of these two equations gives,
\begin{align}
    (R(1/x)+F(1/x))^2-(I(1/x)+G(1/x))^2\delta=S(x,\sqrt{\Delta})\left((R(x)+f(x))^2-(I(x)+g(x))^2\delta\right)\label{change 4}.
\end{align}
\eqref{change 4} is an equation such that the $[x^>]$ and $[x^<]$ degree terms can be separated. We shall factor $S(x,\sqrt{\Delta})=S_+(x)S_-(1/x)$. Then, we get two equations in the form,
\begin{align}
    \begin{split}
        &(R(1/x)+F(1/x))^2-(I(1/x)+G(1/x))^2\delta=NS(1/x)\\
        &(R(x)+f(x))^2-(I(x)+g(x))^2\delta=PS(x).\label{qudratic Dfinite}
    \end{split}
\end{align}
Since the original equation set is equivalent under the transformation $x\to1/x$, we shall have  $F(1/x)=f(1/x)$,  $G(1/x)=g(1/x)$ and $NS(1/x)=PS(1/x)$. This is an algebraic relation between $R(x)$ and $I(x)$ with some non-algebraic coefficients $PS(x)$. However, we currently cannot find another one.

If we denote $R(x)+f(x)\to R(x)$ and $I(x)+g(x)\to I(x)$, \eqref{change 2} and \eqref{change 3} provides a matrix cRBVP in the two dimensional subspace,
\begin{align}
    \left(
\begin{array}{c}
 R(1/x)  \\
 I(1/x) \\
\end{array}
\right)
=
    \left(
\begin{array}{cc}
 a(x) &   b(x)\delta \\
 b(x) & a(x) \\
\end{array}
\right)
    \left(
\begin{array}{c}
 R(x)  \\
 I(x)  \\
\end{array}
\right)+\sqrt{\Delta}
    \left(
\begin{array}{cc}
 s(x) &   r(x)\delta \\
 r(x) & s(x) \\
\end{array}
\right)
    \left(
\begin{array}{c}
 R(x)  \\
 I(x)  \\
\end{array}
\right)\label{change f}.
\end{align}
For \eqref{factor lemma 1} case in \cref{factor lemma}, the results are similar. We only need to change the sign of $\sqrt{\delta}$ on the right hand-side of \eqref{change 3} and \eqref{change 4}. \eqref{qudratic Dfinite} remains the same. The cRBVP in the subspace reads,
\begin{align}
    \left(
\begin{array}{c}
 R(1/x)  \\
 I(1/x) \\
\end{array}
\right)
=
    \left(
\begin{array}{cc}
 a(x) &   b(x)\delta \\
 -b(x) & -a(x) \\
\end{array}
\right)
    \left(
\begin{array}{c}
 R(x)  \\
 I(x)  \\
\end{array}
\right)+\sqrt{\Delta}
    \left(
\begin{array}{cc}
 s(x) &   r(x)\delta \\
 -r(x) & -s(x) \\
\end{array}
\right)
    \left(
\begin{array}{c}
 R(x)  \\
 I(x)  \\
\end{array}
\right).
\end{align}

\begin{remark}
     \eqref{quadratic 4} in the $\lambda_i=\lambda_j$ case plays a similar role as \eqref{qudratic Dfinite} here. 
\eqref{quadratic 4} gives an algebraic relation between $A(x)^2$ and $B(x)^2$, which are equivalent to $R(x),I(x)$ in \eqref{qudratic Dfinite}. We can find a $2\times 2$ matrix cRBVP in the subspace of $(A(x),B(x))$ as well.
\end{remark}

\subsection{Conclusion and Some Discussion for General $n\times n$ Matrix}\label{conclude}
In previous sections, we consider the case of $3\times 3$ matrix cRBVP in the form of \eqref{independent} with condition $Det|P_1(x)|=0$, let us summarize all the integrable cases.
\begin{enumerate}
    \item $\lambda_1=1$ and $\mu_1=0$. We have a linear equation of $\{H_1(x),H_2(x),H_3(x)\}$ with rational coefficients in $x$ and some D-finite $PR_1(x)$ known in $x$.
    \item When $\lambda_2\neq \lambda_3$ and $\lambda_2,\lambda_3$ are rational in $x$, we can find the other two linear equations for $\{H_1(x),H_2(x),H_3(x)\}$ with some D-algebraic coefficients \cite{xu2022interacting} due to the canonical factorization.
    \item When $\lambda_2=\lambda_3$ and the subspace is not diagonalizable, we can also find the other two linear equations of $\{H_1(x),H_2(x),H_3(x)\}$ with some D-algebraic coefficients. 
    \item If $\lambda_2=\lambda_3$ with diagonalizable subspace, and if $k=1-1/\lambda$ satisfies any of the equations \eqref{1},\eqref{2},\eqref{3},\eqref{4}, we find two polynomial equation with one catalytic variable for some linear combination of $\{H_1(x),H_2(x),H_3(x)\}$ with polynomial coefficients.
\end{enumerate}
The integrability of the matrix cRBVP depends on the field extension of $\lambda(x)$. The algebraic property appears in the subspace of conjugate roots. They are determined by the irreducible factor of the characteristic polynomial of either $P_0(x)P_0(1/x)$ or $P_1(x)P_1(1/x)$.

Thus, for general $n\times n$ matrices, we shall consider the characteristic polynomial,
\begin{align}
   \lambda^k\prod_i(\lambda-a_i(x))\times\prod_j(\lambda^2+b_{j1}(x)\lambda+b_{j0}(x))\times\dots\prod_l(\lambda^m+b_{l(m-1)}(x)\lambda^{m-1}+\dots b_{l0}(x))=0\label{character}.
\end{align}
The methodology developed in \cref{RBVP} naturally extends to any irreducible factors of \eqref{character}.

Recall \eqref{au6},
\begin{align}
     (\lambda_i(x)-\lambda_j(1/x))v_i(x)P_1(1/x)u_j(1/x)^T=0.
\end{align}
Due to the $x\to 1/x$ symmetry of the characteristic polynomial, if $\lambda_i(x)$ is not a multiple root, either $\forall j, \lambda_i(x)\neq \lambda_j(1/x)$ and $v_i(x)P_1(1/x)=m_i(1/x)v_i(1/x)$ or $\exists! j,\lambda_i(x)= \lambda_j(1/x)$ and $v_i(x)P_1(1/x)=m_j(1/x)v_j(1/x)$. For subspace of distinct linear factors or Jordan case, we apply the techniques in \cref{different eigenvalues} and \cref{{Jordan Case}} and solve the subspace. For roots conjugated under the Galois group, their corresponding conjugated equations have the form of \eqref{change 2}. The product of all these conjugated equations gives a polynomial equation in the form of \eqref{qudratic Dfinite}.  Four subspaces of multiple roots, \eqref{construction of v(x)} holds. If the multiple roots are from linear factors, for example, $(\lambda-a)^2$, we can apply the discussion in \cref{equal1}, \cref{more algebraic solutions} and \cref{algebra equal case} to find the algebraic structures. Multiple roots from non-linear factors such as $(\lambda^2-a\lambda+b)^2$ require novel techniques beyond current scope.

\section{D-finite Case with Vanished Full Orbit-Sum}\label{orbit example}
In the final part of this paper, we show how the theory of matrix cRBVP applies to some criteria in previous studies.

In the study of quarter-plane lattice walk problems, it is always conjectured that if the orbit sum is $0$, the solution is algebraic. In \cref{RBVP}, we have proved that the orbit-sum is characterized by the null space $v_L(1/x)$ and \eqref{automorphism relation 11}. This equation is a linear equation with an extra term $v_L(1/x)C(x)^T$. If $v_L(1/x)C(x)^T$ is rational, then $[x^0]v_L(1/x)C(x)^T$ is algebraic. In lattice walk problems, we need to have $v_L(1/x)C(x)^T=0$ (since $C(x)^T$ comes from $[y^0]\frac{xy}{K(x,y)}$). In \cite{buchacher2020quadrant}, the author discovered a special lattice walk model whose orbit is zero while the solution is not algebraic. We analyze this model in the cRBVP framework and show what happens in this case.
\subsection{Walks Starting Outside the Quadrant}
In \cite{buchacher2020quadrant}, the author considered a model with allow steps $\{\leftarrow,\rightarrow,\uparrow,\downarrow\}$ in the whole $(i,j)$ plane with restrictions on $i>0$ axis and $j>0$ axis. On the $j>0$ axis, $\leftarrow$ is not allowed and on the $i>0$ axis, $\downarrow$ is not allowed. If we denote $f_{ijn}$ as the number of configuration of n-step paths starting from the point $(-1,-1)$, ending at $(i,j)$, then the generating function $F(x,y,t)=\sum_{i,j,n\geq 0}f_{ijn}x^iy^jt^n$ satisfies the following equation,
\begin{align}
    (1-(x+y+\bar{x}+\bar{y}))F(x,y,t)=\bar{x}\bar{y}-\bar{x}t[y^\geq]F(0,y,t)-\bar{y}t[x^\geq]F(x,0,t).
\end{align}
 \begin{figure}
\centering
\resizebox{0.5\textwidth}{!}{
\begin{tikzpicture}
\tikzset{12/.style={circle, line width=1.5pt, draw=black, fill=green, inner sep=2pt}}
\tikzset{1/.style={circle, line width=1.5pt, draw=black, fill=white, inner sep=2pt}}
\tikzset{2/.style={circle, line width=1.5pt, draw=red, fill=white, inner sep=2pt}}
\tikzset{1x/.style={circle, line width=1.5pt, draw=black, fill=orange, inner sep=2pt}}
\tikzset{12x/.style={circle, line width=1.5pt, draw=black, fill=black, inner sep=2pt}}
\draw [gray, line width=5pt]  (-6,-0.1) -- (6,-0.1);
\draw [gray, line width=5pt]  (-0.1,-6) -- (-0.1,6);
\draw[step=1cm,gray,very thin] (0,0) grid (6,6);
\draw[step=1cm,gray,very thin] (0,0) grid (-6,-6);
\draw[step=1cm,gray,very thin] (0,0) grid (6,-6);
\draw[step=1cm,gray,very thin] (0,0) grid (-6,6);
\begin{scope}[line width=3pt, decoration={markings,mark=at position 0.65 with {\arrow{>}}}]
\draw [postaction=decorate] (-1,-1) node [12x] {} -- (-2,-1);
\draw [postaction=decorate] (-2,-1) node [1] {} -- (-2,-2);
\draw [postaction=decorate] (-2,-2) node [1] {} -- (-3,-2);
\draw [postaction=decorate] (-3,-2) node [1] {} -- (-4,-2);
\draw [postaction=decorate] (-4,-2) node [1] {} -- (-5,-2);
\draw [postaction=decorate] (-5,-2) node [1] {} -- (-5,-1);
\draw [postaction=decorate] (-5,-1) node [1] {} -- (-5,0);
\draw [postaction=decorate] (-5,0) node [1] {} -- (-5,1);
\draw [postaction=decorate] (-5,0) node [1] {} -- (-5,1);
\draw [postaction=decorate] (-5,1) node [1] {} -- (-4,1);
\draw [postaction=decorate] (-4,1) node [1] {} -- (-4,0);
\draw [postaction=decorate] (-4,0) node [1] {} -- (-4,-1);
\draw [postaction=decorate] (-4,-1) node [1] {} -- (-3,-1);
\draw [postaction=decorate] (-3,-1) node [1] {} -- (-3,0);
\draw [postaction=decorate] (-3,0) node [1] {} -- (-3,1);
\draw [postaction=decorate] (-3,1) node [1] {} -- (-2,1);
\draw [postaction=decorate] (-2,1) node [1] {} -- (-1,1);
\draw [postaction=decorate] (-1,1) node [1] {} -- (0,1);
\draw [postaction=decorate] (0,1) node [1] {} -- (1,1);
\draw [postaction=decorate] (1,1) node [1] {} -- (1,2);
\draw  [postaction=decorate] (1,2) node [1] {} -- (2,2);
\draw  [postaction=decorate] (2,2) node [1] {} -- (2,3);
\draw  [postaction=decorate] (2,3) node [1] {} -- (2,4);
\draw  [postaction=decorate] (2,4) node [1] {} -- (3,4);
\draw  [postaction=decorate] (3,4) node [1] {} -- (3,3);
\draw  [postaction=decorate] (3,3) node [1] {} -- (3,2);
\draw  [postaction=decorate] (3,2) node [1] {} -- (4,2);
\draw  (4,2) node [1x] {};

\draw [line width=4pt, ->] (3,-3) -- (3,-2);
\draw [line width=4pt, ->] (3,-3) -- (3,-4);
\draw [line width=4pt, ->] (3,-3) -- (2,-3);
\draw [line width=4pt, ->] (3,-3) -- (4,-3);

\draw [line width=4pt, ->] (3,0) -- (3,1);
\draw [line width=4pt, ->] (3,0) -- (4,0);
\draw [line width=4pt, ->] (3,0) -- (2,0);

\draw [line width=4pt, ->] (0,3) -- (0,2);
\draw [line width=4pt, ->] (0,3) -- (0,4);
\draw [line width=4pt, ->] (0,3) -- (1,3);

\end{scope}
\end{tikzpicture}
}\caption{A example of walk starting outside the quadrant.}
\end{figure}

The kernel $K(x,y)=1-(x+y+\bar{x}+\bar{y})$. It has two roots in $y$,
\begin{small}
\begin{align}
    \begin{split}
        &Y_0(x)=\frac{-\sqrt{\left(-t(x+\bar{x}))+1\right)^2-4 t^2 }-t(x+\bar{x}))+x}{2 t }=t+(x+\bar{x})t^2+(3+x^2+\bar{x}^2)t^3+O(t^4)\\
       &Y_1(x)=\frac{\sqrt{\left(-t(x+\bar{x}))+1\right)^2-4 t^2 }-t(x+\bar{x}))+x}{2 t }=\frac{1}{t}-(x+\bar{x})-t-(x+\bar{x})t^2-(3+x^2+\bar{x}^2)t^3+O(t^4),
    \end{split}
\end{align}
\end{small}
and 
\begin{align}
Y_0(x)Y_1(x)=1, \qquad Y_0(x)+Y_1(x)=x+\bar{x}-1/t.
\end{align}
The symmetry group of $(x,y)$ reads,
\begin{align}
(x,y)\to (1/x,y)\to (1/x,1/y)\to (x,1/y).
\end{align}
It is straight forward to check the orbit sum,
\begin{align}
    xyF(x,y,t)-\bar{x}yF(\bar{x},y,t)+\bar{x}\bar{y}F(\bar{x},\bar{y},t)-x\bar{y}F(x,\bar{y},t)=0\label{out full orbit sum}.
\end{align}
The author use computer experiments to check the properties of $F(1,1,t)$ and conjecture $F(x,y,t)$ is not D-finite. Let us solve this model using the matrix cRBVP.
\subsection{Matrix cRBVP for Walks Starting Outside the Quadrant}\label{outside the quadrant}
The first step is still to consider how many linearly independent equations we need for the problem. We use the same notation as in \cref{Model}. $H$ refers to horizontal and $p,n$ refers to positive and negative. It is not difficult to see that $Hp(x),Hp_{-1}(x)$ generate all functions in the domain $i\geq 0,j\geq 0$. $Hp_{-1}(x)$ itself generate all functions in the domain $i\geq 0,j<0$. $Hn(1/x),Hn_{-1}(1/x)$ generate all functions in $i<0$. Thus, the unknown vector in the matrix cRBVP is chosen as $H(x)=\left(Hn_{-1}(x),Hp_{-1}(x),Hn(x),Hp(x)\right)^T$. We need to find relations between these four generating functions.

Consider the lower half-plane. Denote the generating function of paths ending in the lower half plane as $L(x,y)$, it satisfies,
\begin{align}
    (1-t(x+y+\bar{x}+\bar{y}))L(x,\bar{y})=\bar{x}\bar{y}+t\bar{y}Hn(\bar{x})-tHn_{-1}(\bar{x})-tHp_{-1}(x)\label{out 1}.
\end{align}
$L(x,\bar{y})$ is a formal series in $\bar{y}$. So if we substitute $y=1/Y_0(x)$ into \eqref{out 1}, we have,
\begin{align}
    \bar{x}Y_0(x)+tHn(\bar{x})Y_0(x)-t Hn_{-1}(\bar{x})-tHp_{-1}(x)=0.\label{out 2}
\end{align}
Apply the symmetric transformation $x\to1/x$, we have,
\begin{align}
    xY_0(x)+tHn(x)Y_0(x)-t Hn_{-1}(x)-tHp_{-1}(\bar{x})=0.\label{out 3}
\end{align}
We get two equations.

Then, consider the paths ending in the first quadrant. Denote the generating function of paths ending in the first quadrant as $Ur(x,y)$. It satisfies a functional equation,
\begin{align}
    (1-t(x+y+\bar{x}+\bar{y}))Ur(x,y)=t\bar{y}Hp(x)-t\bar{x}Vp(x)+tHp_{-1}(x)+tVp_{-1}(y)\label{out 4}.
\end{align}
$Ur(x,y)$ is a formal series in $x,y$. It is suitable to substitute $y=Y_0(x)$ into \eqref{out 4}. We consider the two pairs $(x,Y_0(x))$ and $(\bar{x},Y_0(x))$,
\begin{align}
    \begin{split}
        &-\frac{t Hp(x)}{Y_0(x)}+t Hp_{-1}(x)-t\bar{x} Vp\left(Y_0(x)\right)+t Vp_{-1}\left(Y_0\right)=0\\
        &-\frac{t Hp\left(\bar{x}\right)}{Y_0(x)}+t Hp_{-1}\left(\bar{x}\right)-t x Vp\left(Y_0(x)\right)+t Vp_{-1}\left(Y_0\right)=0.\label{out 5}
    \end{split}
\end{align}
Eliminate $Vp(Y_0(x))$ in \eqref{out 5} by a linear combination,
\begin{align}
    \frac{t^2 x Hp(x)}{Y_0}-\frac{t^2 \bar{x}Hp\left(\bar{x}\right)}{ Y_0}+t^2\bar{x} Hp_{-1}\left(\bar{x}\right)-t^2 x Hp_{-1}(x)-t^2(x-\bar{x})Vp_{-1}(Y_0(x))=0.
\end{align}
There is still a $Vp_{-1}(Y_0(x))$ in the equation. We will soon remove it.

Let us consider the left part, the generating functions of lattice paths ending in the second quadrant $Ul(x,y)$. It satisfies a functional equation,
\begin{align}
    (1-t(x+y+\bar{x}+\bar{y}))Ul(\bar{x},y)=-t\bar{y}Hn(\bar{x})+tHn_{-1}(\bar{x})-tVp_{-1}(y)\label{out 6}
\end{align}
$Ul(\bar{x},y)$ is a formal series in $\bar{x},y$. Substitute $y=Y_0(x)$ into \eqref{out 6} and consider the two pairs $(\bar{x},Y_0(x))$ and $(x,Y_0(x))$. We have,
\begin{align}
    \begin{split}
        &-\frac{t Hn\left(\bar{x}\right)}{Y_0}+t Hn_{-1}\left(\bar{x}\right)-t Vp_{-1}\left(Y_0\right)=0\\
        &-\frac{t Hn(x)}{Y_0}+t Hn_{-1}(x)-t Vp_{-1}\left(Y_0\right)=0\label{out 7}
    \end{split}
\end{align}
A linear combination gives the following equation,
\begin{align}
    \frac{t^2 Hn\left(\bar{x}\right)}{Y_0}-\frac{t^2 Hn(x)}{Y_0}+t^2 \left(-Hn_{-1}\left(\bar{x}\right)\right)+t^2 Hn_{-1}(x)=0.\label{out 8}
\end{align}
Notice that both equations of \eqref{out 7} contains $Vp_{-1}(Y_0(x))$. We can combine \eqref{out 7} and \eqref{out 6} to eliminate $Vp_{-1}(Y_0(x))$, which provides us the last equation.
\begin{small}
\begin{align}
    -\frac{t^3 (x-\bar{x}) Hn\left(\bar{x}\right)}{ Y_0}+t^3 (x-\bar{x}) Hn_{-1}\left(\bar{x}\right)-\frac{t^3 x Hp(x)}{Y_0}+\frac{t^3\bar{x} Hp\left(\bar{x}\right)}{ Y_0}+t^3 x Hp_{-1}(x)-t^3 \bar{x}Hp_{-1}\left(\bar{x}\right)=0.\label{out 9}
\end{align}
\end{small}
Now we have find four equations for four unknown functions $Hp(x),Hp_{-1}(x),Hn(x),Hn_{-1}(x)$ and their automorphisms $x\to 1/x$. Combine \eqref{out 2}, \eqref{out 3}, \eqref{out 8}, \eqref{out 9}, we construct a $4\times 4$ matrix cRBVP,
\begin{align}
    \left(\begin{array}{c}
    Hn_{-1}(\bar{x})\\
    Hp_{-1}(\bar{x})\\
    Hn(\bar{x})\\
    Hp(\bar{x})
    \end{array}\right)=M(x)
        \left(\begin{array}{c}
    Hn_{-1}(x)\\
    Hp_{-1}(x)\\
    Hn(x)\\
    Hp(x)
    \end{array}\right)
    +
    C_0^T\label{RBVP for out},
\end{align}
where $M(x)=P_0(x)+\sqrt{\Delta}P_1(x)$,
\begin{align}
    P_0(x)=\left(
\begin{array}{cccc}
 \frac{1}{2} & -\frac{1}{2} & 0 & 0 \\
 -1 & 0 & -\frac{t x^2+t-x}{2 t x} & 0 \\
 0 & 0 & \frac{1}{2} & 0 \\
 \frac{x \left(t x^2+t-x\right)}{2 t} & \frac{x \left(t x^2+t-x\right)}{2 t} & \frac{3 t^2 x^4-2 t^2 x^2+t^2-2 t x^3-2 t x+x^2}{2 t^2 x^2} & x^2 \\
\end{array}
\right),
\end{align}
and
\begin{align}
    P_1(x)=\left(
\begin{array}{cccc}
 -\frac{x \left(t x^2+t-x\right)}{2 \left(t (x-1)^2-x\right) \left(t (x+1)^2-x\right)} & -\frac{x \left(t x^2+t-x\right)}{2 \left(t (x-1)^2-x\right) \left(t (x+1)^2-x\right)} & -\frac{t x^2}{\left(t (x-1)^2-x\right) \left(t (x+1)^2-x\right)} & 0 \\
 0 & 0 & \frac{1}{2 t} & 0 \\
 \frac{t x^2}{\left(t (x-1)^2-x\right) \left(t (x+1)^2-x\right)} & \frac{t x^2}{\left(t (x-1)^2-x\right) \left(t (x+1)^2-x\right)} & \frac{x \left(t x^2+t-x\right)}{2 \left(t (x-1)^2-x\right) \left(t (x+1)^2-x\right)} & 0 \\
 -\frac{x^2}{2 t} & -\frac{x^2}{2 t} & \frac{x-t \left(x^2+1\right)}{2 t^2 x} & 0 \\
\end{array}
\right).
\end{align}
\subsection{Algebraic Properties of Walks Starting Outside the Quadrant}
It is not hard to check the eigenvalues of $P_0(x)P_0(1/x)$ are $\{1,1,1/4,1/4\}$ and the eigenvalues of $P_1(x)P_1(1/x)$ are $\left\{0,0,\frac{3 x^2}{4 \left(t (x-1)^2-x\right) \left(t (x+1)^2-x\right)},\frac{3 x^2}{4 \left(t (x-1)^2-x\right) \left(t (x+1)^2-x\right)}\right\}$. $\mu_1=\mu_2=0$ form a subspace of dimension two. Two linearly independent eigenvectors in the null-space are,
\begin{align}
    \begin{split}
        &v_1(x)=\left(-\frac{x \left(t^2 x^4-2 t^2 x^2+t^2-2 t x^3-2 t x+x^2\right)}{t \left(t x^2+t-x\right)},-\frac{t^2 x^4-2 t^2 x^2-t^2+2 t x^3+2 t x-x^2}{t x \left(t x^2+t-x\right)},0,1\right)\\
        &v_2(x)=\left(\frac{2 t x}{t x^2+t-x},-\frac{t x}{t x^2+t-x},1,0\right).
    \end{split}
\end{align}
$\{v_1(x),v_2(x)\}$ span the the subspace. Multiply $v_1(x),v_2(x)$ to the left to \eqref{RBVP for out}, we have,
\begin{align}
    \begin{split}
    &-2 (x^2-1)Hn(x)+Hp\left(\bar{x}\right)-x^2 Hp(x)-\frac{Hn_{-1}\left(\bar{x}\right) \left(t x^2-2 t x+t-x\right) \left(t x^2+2 t x+t-x\right)}{t \left(t (x+\bar{x})-1\right)}\\
    &-\frac{Hn_{-1}(x) \left(t x^2-t+x\right) \left(3 t x^2+t-x\right)}{t x^2 \left(t(x+\bar{x})-1\right)}-\frac{Hp_{-1}\left(\bar{x}\right) \left(t^2 x^4-2 t^2 x^2-t^2+2 t x^3+2 t x-x^2\right)}{t x^2 \left(t(x+\bar{x})-1\right)}\\
    &-\frac{Hp_{-1}(x) \left(t^2 x^4+t^2-2 t x^3-2 t x+x^2\right)}{t \left(t (x+\bar{x})-1\right)}-\frac{(x-1) x (x+1)}{t}\\
    &+\frac{x(x-\bar{x}) \sqrt{\left(-t x^2-t+x\right)^2-4 t^2 x^2}}{t \left(t(x+\bar{x})-1\right)}=0\label{out v1}
    \end{split},
\end{align}
and
\begin{align}
    \begin{split}
        &+Hn\left(\bar{x}\right)-Hn(x)+\frac{2 t  Hn_{-1}\left(\bar{x}\right)}{t(x+\bar{x})-1}-\frac{2 tHn_{-1}(x)}{t(x+\bar{x})-1}+\frac{t Hp_{-1}(x)}{t(x+\bar{x})-1}-\frac{t  Hp_{-1}\left(\bar{x}\right)}{t(x+\bar{x})-1}\\
        &+\frac{(x-\bar{x})  \sqrt{\left(-t x^2-t+x\right)^2-4 t^2 x^2}}{2 t  x\left(t (x+\bar{x}))-1\right)}-\frac{ (x-\bar{x})}{2 t }=0\label{out v2}.
    \end{split}
\end{align}
The $[x^>]$ terms of both \eqref{out v1} and \eqref{out v2} give linear relation between $Hp(x),Hp_{-1}(x),Hn(x),Hn_{-1}(x)$ with an extra D-finite term $PR(x)$. Notice that $\bar{x}\times \eqref{out v1}-2 x\times \eqref{out v2}$ gives a linear relation without $\sqrt{\Delta}$,
\begin{align}
    \begin{split}
        &\frac{Hn_{-1}(x) \left(t(x+\bar{x})-1\right)}{t x}-\bar{x}\frac{Hn_{-1}\left(\bar{x}\right) \left(t (x+\bar{x})-1\right)}{t}-2 x Hn\left(\bar{x}\right)+2 \bar{x}Hn(x)\\
        &+\frac{Hp_{-1}\left(\bar{x}\right) \left(t(x+\bar{x})-1\right)}{t x}-\frac{Hp_{-1}(x) \left(t(\bar{x}+x)-1\right)}{tx}+\bar{x}Hp\left(\bar{x}\right)-x Hp(x)=0
    \end{split}
\end{align}
This is the reason of zero orbit-sum \eqref{out full orbit sum}. Since the null-space of $P_1(x)P_1(1/x)$ is dimension two, we can find a vector $v(x)$ inside this subspace such that $v(x)C_0(x)^T=0$. But for general vectors in the null-space, for example \eqref{out v1} and \eqref{out v2} ,$v(x)C_0(x)^T\neq0$.

The subspace with eigenvalue $1/4$ is also integrable. The eigenvectors corresponding to $\lambda=1/4$ is,
\begin{align}
    \begin{split}
        &v_3(x)=(1,1,0,0)\\
        &v_4(x)=(0,0,1,0).
    \end{split}
\end{align}
By the discussion in \cref{RBVP}, we immediately know for any $v(x)$ in this subspace,

$\left(v(\bar{x})+\frac{1}{m_0(\bar{x})}v(x)P_0(\bar{x})\right)H(x)^T$ or, $Hn(\bar{x})$ satisfies a polynomial equation of degree $3$ with algebraic coefficients.
\section{Final Comments}
The main objective of this paper is to establish the theory of matrix cRBVP in the framework of analytic combinatorics and to solve some non-Weyl walks avoiding a quadrant. We call an $n\times n$ matrix cRBVP integrable if it can be reduced to $n$ independent equations in one variable\footnote{We need to find $n-1$ independent equations in advance, the last one can be obtained by the matrix. There are $n$ independent equations in all.}. These equations can be linear or polynomial equations with one catalytic variable. The integrability condition of this matrix cRBVP depends on the eigenvalues and eigenvectors of some associated matrix, which is concluded in \cref{conclude}. 

There are some discussions about what shall be done next.
\begin{enumerate}
    \item Our theory applies to various 2-D lattice walk problems, including $M$-quadrant cones or weighted walks. We have also checked the model discussed in \cite{bousquet2016square} and add weights to the steps meeting the boundaries with the restriction that the orbit-sum condition is satisfied. In our theory, orbit-sum condition means the associated matrix $P_0(x,t)$ has an eigenvalue $1$ (or $P_1(x,t)$ has an eigenvalue $0$). However, for all the lattice walk models we have calculated, this also guarantees $P_0(x,t)P_0(1/x,t)$ has an eigenspace with a double-root eigenvalue $1/4$. We wonder whether this is a universal phenomenon and want to understand what causes this phenomenon. 
    \item We discussed how to resolve a matrix cRBVP to several polynomial equations with one catalytic variable (possibly with some extra D-finite terms). For the model to be explicitly solvable, we need to solve these polynomial equations. Recall \eqref{polyproof matrix}. For a polynomial equation $P(Q(x,t),Q_1,Q_2\dots Q_k,t,x)=0$, if we find $k$ distinct $X_i$ such that the following equations hold
\begin{align}
\begin{split}
&P(Q(X_i,t),Q_1,Q_2\dots Q_k,t,X_i)=0\\
&\partial_{x_0}P(Q(X_i,t),Q_1,Q_2\dots Q_k,t,X_i)=0\\
&\partial_x P(Q(X_i,t),Q_1,Q_2\dots Q_k,t,X_i)=0,\label{dis}
\end{split}
\end{align}
and the determinant of the Jacobi matrix defined by \eqref{Jacob} is not zero, we can explicitly solve \eqref{dis}. However, do we always have enough $X_i$ for a well-defined problem? Our recent work \cite{xu2023combinatorial} shows that this depends on some parameters for scalar cRBVP. We conjectured this also happens in the polynomial equations obtained from the matrix cRBVP.
\item Is it possible to extend the idea of solving matrix cRBVP back to the analytic insight? The main obstruction we met in solving the cRBVP is that the irreducible factor in the characteristic polynomial introduce a $\sqrt{\delta}$ which is not analytic either near $x=0$ nor $x=\infty$. It is not suitable for the positive term extraction. In the original Riemann boundary value problem with Carleman shift, we are dealing with functions analytic on some curve. $x,1/x$ corresponds to $s,\sigma(s)$ on some curve $s\in X_0(y_1y_2)$ (see \cite{raschel2012counting} for details). Intuitively speaking, we find an matrix RBVP on a curve,
\begin{align}
    \begin{pmatrix}
        A_0(\sigma (s))\\
        A_1(\sigma(s))
    \end{pmatrix}
    =\Big(P_0(y(s))+(s-\sigma(s) )P_1(y(s))\Big)
     \begin{pmatrix}
        A_0(s)\\
        A_1(s)     
    \end{pmatrix}
    +C(s)^T\label{matrix RBVP},
\end{align}
The matrix is reformulated in the symmetric part and the antisymmetric part under $s\to\sigma(s)$. The quadratic terms of $s,\sigma(s)$ disappear in the coefficient matrix due to the quadratic relation between $s,y(s)$. Since $y(s)=y(\sigma(s))$, $P_0(y(s))$ and $P_1(y(s))$ are analytic on the curve. Then our job becomes finding vectors $v(\sigma(s),y(s)),u(s,y(s))$ such that,
\begin{align}
    v(\sigma(s),y)(P_0(y(s))+(s-\sigma(s))P_1(y(s)))=\lambda(s) u(s,y).\label{RBVP ma}
\end{align}

However, the setting of RBVP and cRBVP are different. There are some difference in the coefficient matrix \eqref{RBVP ma} and we do not know whether we can do this exactly. 
\item In mathematical physics, Birkhoff factorization is associated with some loop group structures and it is generally called Birkhoff decomposition. For example, Let $\phi$ be a algebra homomorphism $Hom(H,A)$. Let $K$ be a matrix algebra and $A=K[x^{-1}][[x]]$ be the algebra of Laurent series. In the theory of connected filtered cograded Hopf algebra \cite{guo2008algebraic}, there are unique linear maps, $\phi_-:H\to K[x^{-1}]$, $\phi_+:H \to K[[x]]$, such that
\begin{align}
    \phi=\phi_-^{*(-1)}*\phi_+.
\end{align}
$\phi_-^{*(-1)}$ is defined by antipode, see \cite{guo2008algebraic} for detailed definitions. Then, for any $h\in H$ (Hopf algebra), $\phi(h)$ admits a Birkhoff factorization in matrix form. I wonder whether we can find some algebra criteria for the results we find in this paper, especially for the double root cases. Roughly speaking, we are not solving the matrix cRBVP by factorization but by reducing it to polynomial equations with catalytic variables. The space of $H_1(1/x),H_2(1/x),H_3(1/x)$ is not related to $H_1(x),H_2(x),H_3(x)$ as a vector space homomorphism but a map between polynomial rings.
\end{enumerate}
\section*{Acknowledgment}
This work is supported by Beijing Institute of Mathematical Sciences and Applications (BIMSA).\\
During the preparation of this work, Ruijie Xu used Deepseek in order to improve the writing. After using this tool/service, Ruijie Xu reviewed and edited the content as needed and takes full responsibility for the content of the published article.
\bibliographystyle{alpha}
\bibliography{reportbib}
\end{document}